\documentclass[10pt]{amsart}

\usepackage{amsmath}
\usepackage{amssymb}
\usepackage{verbatim}
\usepackage{enumerate}

\usepackage{setspace}
\usepackage{fancyhdr}
\usepackage{indentfirst}
\usepackage{amsthm}
\usepackage{mathrsfs}
	
\theoremstyle{definition}
\newtheorem{thm}{Theorem}[section]
\newtheorem{lem}[thm]{Lemma}
\newtheorem{prp}[thm]{Proposition}
\newtheorem{dfn}[thm]{Definition}
\newtheorem{cor}[thm]{Corollary}
\newtheorem{cnj}[thm]{Conjecture}

\newtheorem{ntn}[thm]{Notation}

\newcommand{\bt}{\beta}
\newcommand{\gm}{\gamma}
\newcommand{\dt}{\delta}

\newcommand{\eps}{\varepsilon}
\newcommand{\zt}{\zeta}

\newcommand{\te}{\theta}

\newcommand{\sm}{\sigma}

\newcommand{\ph}{\varphi}

\newcommand{\om}{\omega}

\newcommand{\Z}{{\mathbb{Z}}}

\newcommand{\N}{{\mathbb{N}}}

\newcommand{\sint}{{\mathrm{int}}}
\newcommand{\diam}{{\mathrm{diam}}}
\newcommand{\dist}{{\mathrm{dist}}}

\newcommand{\supp}{{\mathrm{supp}}}

\newcommand{\card}{{\mathrm{card}}}

\newcommand{\invlim}{\displaystyle \lim_{\longleftarrow}}

\newcommand{\ts}[1]{{\textstyle{#1}}}

\newcommand{\set}[1]{\left\{#1\right\}}

\newcommand{\del}{\partial}

\newcommand{\JS}{\mathcal{Z}}

\title[Smallness and Comparison]{Smallness and Comparison 
Properties for Minimal Dynamical Systems}

\author{Julian Buck}

\address{Department of Mathematics, Francis Marion University, 
Florence, SC 29502}

\date{27 June 2013}

\begin{document}

\begin{abstract}
We introduce the dynamic comparison property for minimal 
dynamical systems which has applications to the study of crossed 
product $C^{*}$-algebras. We demonstrate that this property holds 
for a large class of systems which includes all examples where the 
underlying space is finite-dimensional, as well as for an explicit 
infinite-dimensional example, showing that the it is strictly weaker 
than finite-dimensionality in general.
\end{abstract}

\maketitle 

\section{Introduction}\label{SectionIntro}

\indent

The role of topological dynamics in the classification program for 
nuclear $C^{*}$-algebras enjoys a rich and long history. 
Transformation group $C^{*}$-algebras associated to Cantor 
minimal systems were studied and classified in \cite{Put} and 
\cite{GPS}. The case for general finite-dimensional compact metric 
spaces was analyzed in \cite{QLinPhDiff}, \cite{HLinPh}, and 
\cite{TomsWinter}, where it is shown that the transformation group 
$C^{*}$-algebra $C^{*}(\Z,X,h)$ is classifiable so long as 
projections separate traces. In \cite{TomsWinter}, Toms and 
Winter additionally prove that $C^{*}(\Z,X,h)$ is stable under 
tensoring with the Jiang-Su algebra $\JS$, assuming only the 
finite-dimensionality of $X$.

Very little is known in the case where $X$ is an 
infinite-dimensional space. Giol and Kerr \cite{GK} have 
constructed examples of infinite-dimensional minimal dynamical 
systems such that the associated transformation group 
$C^{*}$-algebras have perforation in their K-theory and Cuntz 
semigroups (such examples lie outside the $\JS$-stable class). 
The existence of such minimal systems demonstrates the 
intractability of classifying $C^{*}$-algebras associated to 
general infinite-dimensional dynamical systems and indicates 
the need for regularity properties which rule out the type of 
behavior exhibited by their examples. In this paper we survey 
some existing regularity properties, and introduce a new one 
which holds for all finite-dimensional minimal systems as well 
as for at least some infinite-dimensional examples. Applications 
to crossed product $C^{*}$-algebras are referenced which will 
appear in subsequent papers.

We would like to thank N. Christopher Phillips for his numerous 
suggestions and insights, as much of it was completed at the 
University of Oregon under his supervision as part of the author's 
Ph.D. thesis. We would also like to thank Taylor Hines, David 
Kerr, and Ian Putnam for several helpful conversations and 
ideas related to aspects of this paper.

\section{Preliminaries}\label{SectionPreliminaries}

\begin{ntn}\label{MinDynSys}
Throughout, we let $X$ be an infinite compact metrizable space, 
and let $h \colon X \to X$ be a minimal homeomorphism. The 
corresponding minimal dynamical system $(X,h)$ will sometimes 
be denoted simply by $X$, with the homeomorphism $h$ 
understood. We denote by $M_{h}(X)$ the space of $h$-invariant 
Borel probability measures on $X$. Whenever necessary, it will be 
assumed that $X$ is a metric space with metric $d$. In this case, 
for $x \in X$ and $\eps > 0$, we will denote the $\eps$-ball 
centered at $x$ by 
\[
B(x,\eps) = \set{y \in X \colon d(x,y) < \eps}.
\]
We denote the boundary of a set $A \subset X$ by $\del A$. In 
particular, if $U \subset X$ is open then $\del U = \overline{U} 
\setminus U$, and if $C \subset X$ is closed then $\del C = C 
\setminus \sint (C)$.
\end{ntn}

\begin{dfn}\label{UnivNull}
A Borel set $E \subset X$ is called {\emph{universally null}} if 
$\mu (E) = 0$ for all $\mu \in M_{h}(X)$.
\end{dfn}

\begin{lem}\label{OpenPositiveMeas}
Let $(X,h)$ be as in Notation \ref{MinDynSys}. 
If $U \subset X$ is open and non-empty, then $\mu (U) > 0$ for all 
$\mu \in M_{h}(X)$. Moreover, if $C \subset X$ is closed and $\mu (C) 
< \mu (U)$ for all $\mu \in M_{h}(X)$, then $\inf_{\mu \in M_{h}(X)} 
[\mu (U) - \mu (C)] > 0$. In particular, $\inf_{\mu \in M_{h}(X)} (U) > 0$ 
for any non-empty open set $U \subset X$.
\end{lem}

\begin{proof}
%

It is well-known that if $U \subset X$ is open and non-empty, then 
the minimality of $h$ implies that $X = \bigcup_{n=-\infty}^{\infty} 
h^{n}(U)$. Suppose that $\mu (U) = 0$ for some $\mu \in M_{h}(X)$. 
Then the $h$-invariance of $\mu$ implies that
\[
1 = \mu (X) = \mu \left( \bigcup_{n= -\infty}^{\infty} h^{n}
(U) \right)  \leq \sum_{n= -\infty}^{\infty} \mu (h^{n}(U)) =
\sum_{n= -\infty}^{\infty} \mu (U) = 0,
\]
a contradiction. Now, the map $\gm_{U} \colon M_{h}(X) \to 
[0,1]$ given by $\gm_{U}(\mu) = \mu (U)$ is a 
lower-semicontinuous function by Proposition 2.7 of 
\cite{ParToms}, and strictly positive by the above argument. 
By an entirely analogous argument as in the proof of 
Proposition 2.7 of \cite{ParToms}, if $C \subset X$ is closed 
the map $\gm_{C} \colon M_{h}(X) \to [0,1]$ given by $\gm_{C} 
(\mu) = \mu (C)$ is upper-semicontinuous. Define $\gm_{U,C} 
\colon M_{h}(X) \to [0,1]$ by $\gm_{U.C} (\mu) = \gm_{U} (\mu) 
- \gm_{C} (\mu) = \mu (U) - \mu (C)$. Since the sum of 
upper-semicontinuous functions is upper-semicontinuous, 
and the negative of a positive upper-semicontinuous function 
is lower-semicontinuous, it follows that $\gm_{U,C} = - 
([-\gm_{U}] + \gm_{C})$ is lower-semicontinuous. By 
assumption, $\gm_{U,C}$ is strictly positive. The compactness 
of $M_{h}(X)$ implies that $\gm_{U,C}$ achieves a lower bound 
$\dt > 0$ on $M_{h}(X)$. But then $\inf_{\mu \in M_{h}(X)} [\mu 
(U) - \mu (C)] \geq \dt > 0$, as claimed. The final observation 
follows immediately from the previous one by taking $C = 
\varnothing$.
\end{proof}

\begin{lem}\label{ArbSmallOpenSet}
Let $(X,h)$ be as in Notation \ref{MinDynSys}. 
For any $\eps > 0$ and any closed set $F \subset X$ with 
$F$ universally null, there is a non-empty 
open set $E \subset X$ such that $F \subset E$ and $\mu (E) 
< \eps$ for all $\mu \in M_{h}(X)$.
\end{lem}

\begin{proof}
Define a sequence $(E_{n})_{n=0}^{\infty}$ of open sets by
$E_{n} = \set{x \in X \colon \dist (x,F) < 1/n}$. Then 
$\overline{E}_{n+1} \subset E_{n}$ for all $n \in \N$, and 
$\bigcap_{n=0}^{\infty} E_{n} = F$. Choose continuous functions 
$f_{n} \colon X \to [0,1]$ with $f_{n} = 1 $ on $\overline{E}_{n+1}$ 
and $\supp (f_{n}) \subset E_{n}$. Then $f_{n} \geq f_{n+1}$ for all 
$n \in \N$. Now each $f_{n}$ defines an affine function 
$\widehat{f}_{n}$ on $M_{h}(X)$
by
\[
\widehat{f}_{n}(\mu) = \int_{X} f_{n} \; d\mu.
\]
Applying the Dominated Convergence Theorem, we conclude that
\[
\lim_{n \to \infty} \widehat{f}_{n} (\mu) = \lim_{n \to \infty}
\int_{X} f_{n} \; d \mu = \int_{X} \lim_{n \to \infty} f_{n} \; d
\mu = \mu (F) = 0
\]
for all $\mu \in M_{h}(X)$. It follows that the monotone decreasing
sequence $(\widehat{f}_{n})_{n=1}^{\infty}$ of continuous functions 
converges pointwise to the continuous affine function $\widehat{f} = 
0$ on the compact set $M_{h}(X)$, and so Dini's Theorem implies that 
the convergence is uniform. Therefore, there is an $N \in \N$ such 
that $\widehat{f}_{N} (\mu) < \eps$ for all $\mu \in M_{h}(X)$. 
Finally, set $E = E_{N+1}$. Then $F \subset E$, and $f_{N} 
\vert_{\overline{E}} = 1$ implies that
\[
\mu (E) \leq \int_{X} f_{N} \; d \mu = \widehat{f}_{N} (\mu) <
\eps
\]
for all $\mu \in M_{h}(X)$.
\end{proof}

\begin{cor}\label{MeasApproxZeroBdy}
Let $(X,h)$ be as in Notation \ref{MinDynSys}.
\begin{enumerate}
\item\label{SmallOpenInsideOpen} For any $\eps > 0$ and any 
non-empty open set $U \subset X$, there is a non-empty open set $E 
\subset U$ such that $\mu (E) < \eps$ for all $\mu \in M_{h}(X)$.
\item\label{ClosedInsideOpen} For any $\eps > 0$ and any non-empty 
open set $U \subset X$ with $\del U$ universally null, there is a closed 
set $K \subset U$ with $\sint (K) \neq \varnothing$ such that $\mu 
(U \setminus K) < \eps$ for all $\mu \in M_{h}(X)$.
\item\label{OpenInsideOpen} For any $\eps > 0$ and any non-empty 
open set $U \subset X$ with $\del U$ universally null, there is an open 
set $E \subset U$ with $\overline{E} \subset U$ such that $\mu (U 
\setminus \overline{E}) < \eps$ for all $\mu \in M_{h}(X)$.
\item\label{OpenOutsideClosed} For any $\eps > 0$ and any closed set 
$K \subset X$ with $\del K$ universally null, there is an open set $E 
\subset X$ such that $K \subset E$ and $\mu (E \setminus K) < \eps$ 
for all $\mu \in M_{h}(X)$.
\end{enumerate}
\end{cor}

\begin{proof}
\begin{enumerate}
\item Fix any point $x_{0} \in U$, and let $F = \set{x_{0}}$, which 
is easily seen by minimality to satisfy $\mu (F) = 0$ for all $\mu \in 
M_{h}(X)$. Let $\dt > 0$ be such that $B(x_{0},\dt) \subset U$, and 
replace the sets $E_{n}$ in the proof of Lemma 
\ref{MeasApproxZeroBdy} by $E_{n} = B(x_{0},\dt/n)$. This ensures 
that the set $E = E_{N+1}$ satisfies $E \subset U$.
\item Let $F = \del U$, let $E = E_{N+1}$ be as in the proof of 
Lemma \ref{MeasApproxZeroBdy}, but with $N$ chosen so large 
that $X \setminus \overline{E} \neq \varnothing$, and set $K = 
\overline{U \cap (X \setminus \overline{E})}$. Note that $K \subset 
U$ since $\dist (X \setminus \overline{E}, X \setminus U) > 1/(N+1)$, 
and that $\sint(K) = U \cap (X \setminus \overline{E}) \neq \varnothing$. 
Finally, $U \setminus K \subset U \setminus (U \cap (X \setminus E) ) 
\subset E$ implies $\mu (U \setminus K) \leq \mu (E) < \eps$ for all 
$\mu \in M_{h}(X)$.
\item This is a convenient restatement of 
(\ref{ClosedInsideOpen}), with $\sint (K) = E$ and $K = \overline{E}$.
\item Let $F = \del K$, let $E$ be as in the proof of Lemma 
\ref{MeasApproxZeroBdy}, and set $U = K \cup E$. Then $U$ is open 
since $U = \sint (K) \cup \del (K) \cup E = \sint (K) \cup E$ (as $\del 
(K) \subset E$), and $\mu (U \setminus K) = \mu (E) < \eps$ for all 
$\mu \in M_{h}(X)$.
\end{enumerate}
\end{proof}

\section{Topologically Small and Thin Sets}

The following theorem is the well-known Rokhlin tower construction, 
where the space $X$ is decomposed in terms of a closed set $Y 
\subset X$ and the ``first return times to $Y$'' for the points of $X$. 
We show that a Rokhlin tower can be made compatible with some 
given partition of $X$ by sets with non-empty interior, in the sense 
that the interior of each level in the tower is contained in exactly 
one set of the partition.

\begin{thm}\label{RokhlinTower}
Let $(X,h)$ be as in Notation \ref{MinDynSys}. 
Let $Y \subset X$ be a closed set with $\sint (Y) \neq \varnothing$.
For $y \in Y$, define $r(y) = \min \set{m \geq 1 \colon h^{m}(y) \in 
Y}$. Then $\sup_{y \in Y} r(y) < \infty$, so there are finitely many 
distinct values $n(0) < n(1) < \cdots < n(l)$ in the range of $r$. 
For $0 \leq k \leq l$, set
\[
Y_{k} = \overline{ \set{ y \in Y \colon r(y) = n(k) } } 
\hspace{0.5 in} \textnormal{and} \hspace{0.5 in} 
Y_{k}^{\circ} = \sint (\set{ y \in Y \colon r(y) = n(k) }).
\]
Then:
\begin{enumerate}
\item the sets $h^{j}(Y_{k}^{\circ})$ are pairwise disjoint for $0 \leq
k \leq l$ and $0 \leq j \leq n(k) - 1$;
\item $\bigcup_{k=0}^{l} Y_{k} = Y$;
\item $\bigcup_{k=0}^{l} \bigcup_{j=0}^{n(k)-1} h^{j}(Y_{k}) = X$.
\end{enumerate}
Moreover, given any finite partition $\mathcal{P}$ of $X$
(consisting of sets with non-empty interior), there exist closed
sets $Z_{0},\ldots,Z_{m} \subset Y$ and non-negative integers $t(0) 
\leq t(1) \leq \cdots \leq t(m)$ such that with $Z_{k}^{(0)} = 
Z_{k} \setminus \del Z_{k}$ (which may be empty) for $0 \leq k \leq 
m$, we have:
\begin{enumerate}
\item the sets $h^{j}(Z_{k}^{(0)})$ are pairwise disjoint for $0
\leq k \leq m$ and $0 \leq j \leq t(k) - 1$;
\item $\bigcup_{k=0}^{m} Z_{k} = Y$;
\item $\bigcup_{k=0}^{m} \bigcup_{j=0}^{t(k)-1} h^{j}(Z_{k}) = X$;
\item for $0 \leq k \leq m$ and $0 \leq j \leq t(k) - 1$, the set
$h^{j}(Z_{k}^{(0)})$ is contained in exactly one $P \in 
\mathcal{P}$.
\end{enumerate}
\end{thm}

\begin{proof}
The finiteness of $r(y)$ and all statements concerning the sets
$Y_{k}$ are shown in \cite{QLinPh1}, \cite{QLinPh2}, and 
\cite{QLinPhDiff} (as well as other places). Now suppose we have 
a finite partition $\mathcal{P}$ of $X$ consisting ofsets with 
non-empty interior. For each $0 \leq k \leq l$, the set
\[
\mathcal{B}_{k} = \set{ h^{-j} \left( h^{j}(Y_{k}) \cap P \right) 
\colon 0 \leq j \leq n(k)-1, P \in \mathcal{P} }
\]
is a cover of $Y_{k}$ by a finite collection of sets with non-empty 
interior. Write $\mathcal{B}_{k} = \set{B_{1},\ldots,B_{N}}$ for an 
appropriate choice of $N \in \N$. Let $\mathcal{C}_{k}$ be the 
collection of all sets of the form $D = \bigcap_{i=1}^{m} C_{i}$, 
where each for each $i$, there is a $j \in \set{1,\ldots,N}$ such 
that either $C_{i} = B_{j}$ or $C_{i} = Y_{k} \setminus B_{j}$. Set 
$\mathcal{C}^{\circ} = \bigcup_{k=0}^{l} \mathcal{C}_{k}$ and 
$\mathcal{C} = \set{ \overline{D} \colon D \in \mathcal{C}^{\circ}}$, 
both of which are finite collections of sets. Write $\mathcal{C} = 
\set{Z_{0}',\ldots,Z_{m}'}$, and for $0 \leq i \leq 
m$, set $t(i) = n(k)$ where $Z_{i}' = \overline{D}$ and $D \in 
\mathcal{C}_{k}$. Without loss of generality, arrange the order of 
the sets $Z_{0}',\ldots,Z_{m}'$ so that $t(0) \leq t(1) \leq \cdots 
\leq t(m)$. Finally, define $Z_{k}$ and $Z_{k}^{(0)}$ for $0 \leq k 
\leq m$ by
\[
Z_{0} = Z_{0}', \hspace{0.5 in} Z_{k} = \overline{ Z_{k}'
\setminus \textstyle{\bigcup_{j=0}^{k-1} Z_{j} }}, \hspace{0.5 in} 
Z_{k}^{(0)} = Z_{k} \setminus \del Z_{k}.
\]
Then $Z_{0},\ldots,Z_{m}$ is a cover of $Y$ by closed sets with the 
desired properties.
\end{proof}

In applications of the Rokhlin tower construction to $C^{*}$-algebras, 
it is often technically important to have some control over the 
boundary $\del Y$ of the closed set $Y \subset X$ used as the base. 
In \cite{HLinPh} the sets employed are taken to have universally null 
boundaries, but this restriction will be too weak for our purposes. 
Instead, we need to insist the boundaries of the sets used be small 
in a more topological sense. In \cite{QLinPhDiff} this is accomplished 
by restricting to the situation where $X$ is a compact smooth 
manifold and $h$ is a minimal diffeomorphism, then requiring that 
$\del Y$ satisfy a certain transversality condition. Definition 
\ref{TopologicallySmall} that follows, which first appeared in 
\cite{PhTopSmall}, is an attempt to formulate an analogous property 
for the case of a more general compact metric space. For our 
purposes, we will often find it convenient to use another form of a 
smallness property for closed sets, which is given in Definition 
\ref{ThinSet}. The connection between these two definitions is 
considered in Proposition \ref{TopSmallIsThin}.

\begin{dfn}\label{TopologicallySmall}
Let $(X,h)$ be as in Notation \ref{MinDynSys}. A closed subset 
$F \subset X$ is said to be {\emph{topologically $h$-small}} if 
there is some $m \in \Z_{+}$ such that whenever $d(0), d(1), 
\ldots, d(m)$ are $m+1$ distinct elements of $\Z$, then 
$h^{d(0)}(F) \cap h^{d(1)}(F) \cap \cdots \cap h^{d(m)}(F) = 
\varnothing$. The smallest such constant $m$ is called the 
{\emph{topological smallness constant}}.
\end{dfn}

We assemble some basic results about topologically $h$-small 
sets.

\begin{lem}\label{TopSmallProps}
Let $(X,h)$ be as in Notation \ref{MinDynSys}.
\begin{enumerate}
\item\label{TopSmallBasic} Let $F \subset X$ be topologically 
$h$-small with topological smallness constant $m$, let $K \subset 
F$ be closed, and let $d \in \Z$. Then $K$ is topologically $h$-small 
with smallness constant at most $m$, and $h^{d}(F)$ is topologically 
$h$-small with smallness constant $m$.
\item\label{TopSmallIntersection} The intersection of arbitrarily 
many topologically small sets is topologically small.
\item\label{TopSmallUnion} Let $F_{1},\ldots,F_{n} \subset X$ be 
topologically $h$-small, where $F_{j}$ has topological smallness 
constant $m_{j}$ for $1 \leq j \leq n$. Then $F = \bigcup_{j=1}^{n} 
F_{j}$ is topologically $h$-small with smallness contant $m = 
\sum_{j=1}^{n} m_{j}$.
\end{enumerate}
\end{lem}

\begin{proof}
\begin{enumerate}
\item This is immediate from the definition.
\item This follows immediately from (\ref{TopSmallBasic}), since 
the intersection of arbitrarily many closed sets is a closed subset 
of any one of them.
\item With $m = \sum_{j=1}^{n} m_{j}$, let $d(0),\ldots,d(m) \in \Z$ 
be $m+1$ distinct integers. Let $\mathcal{S}$ be the collection 
of all functions $s \colon \set{0,\ldots,m} \to \set{1,\ldots,n}$. By the 
pigeonhole principle, for any $s \in \mathcal{S}$ there is a $j \in 
\set{1,\ldots,n}$ such that $\card (s^{-1}(j)) > m_{j}$. Since $F_{j}$ 
has topological smallness constant $m_{j}$, it follows that 
$\bigcap_{i \in s^{-1}(j)} h^{i}(F_{j}) = \varnothing$. This implies 
that 
\[
\ts{ \bigcap_{j=0}^{m} h^{d(j)}(F_{s(j)}) \subset \bigcap_{i \in 
s^{-1}(j)} h^{i}(F_{j}) = \varnothing,}
\]
which gives 
\[
\ts{ \bigcap_{j=0}^{m} h^{d(j)} (F) = \bigcup_{s \in \mathcal{S}} 
\bigcap_{j=0}^{m} h^{d(j)} (F_{s(j)}) = \varnothing,}
\]
as required.
\end{enumerate}
\end{proof}

\begin{dfn}\label{ThinSet}
Let $(X,h)$ be as in Notation \ref{MinDynSys}. 
Let $F \subset X$ be closed and let $U \subset X$ be open. We write
$F \prec U$ if there exist $M \in \N$, $U_{0},\ldots,U_{M} \subset X$
open, and $d(0),\ldots,d(M) \in \Z$ such that:
\begin{enumerate}
\item $F \subset \bigcup_{j=0}^{M} U_{j}$;
\item $h^{d(j)}(U_{j}) \subset U$ for $0 \leq j \leq M$;
\item the sets $h^{d(j)}(U_{j})$ are pairwise disjoint for $0 \leq j
\leq M$.
\end{enumerate}
We say the closed set $F$ is {\emph{thin}} if $F \prec U$ for every
non-empty open set $U \subset X$.
\end{dfn}

It is clear that any closed subset of a thin set is thin, and hence 
the intersection of arbitrarily many thin sets is thin. It is also 
clear that if $F$ is thin, then so is $h^{n}(F)$ for any $n \in \Z$.

\begin{lem}\label{ThinHasThinNbd}
Let $(X,h)$ be as in Notation \ref{MinDynSys}. 
Suppose that $F \subset X$ is closed and $U \subset X$ is open with
$F \prec U$. Then there is an open set $V \subset X$ such that $F
\subset V$ and $\overline{V} \prec U$.
\end{lem}

\begin{proof}
Since $F \prec U$, there exist $M \in \N$, $U_{0},\ldots,U_{M} 
\subset X$ open, and $d(0),\ldots,d(M) \in \Z$ such that $F \subset
\bigcup_{j=0}^{M} U_{j}$ and such that the sets $h^{-d(j)}(U_{j})$
are pairwise disjoint subsets of $U$. Let $E = \bigcup_{j=0}^{M}
U_{j}$, and use $X$ locally compact Hausdorff to choose an open set 
$V$ with $\overline{V}$ compact satisfying $F \subset V \subset 
\overline{V} \subset E$. Then $\overline{V} \prec U$ using the same 
open sets $U_{j}$ and integers $d(j)$ as for $F$.
\end{proof}

\begin{lem}\label{ThinHasMeasZero}
Let $(X,h)$ be as in Notation \ref{MinDynSys}. If $F \subset X$ is thin, 
then $F$ is universally null.
\end{lem}

\begin{proof}
Let $\eps > 0$ be given, let $\mu \in M_{h}(X)$, and choose $N \in \N$ 
such that $1/N < \eps$. Since the action of $h$ on $X$ is free, there 
is a point $x \in X$ such that $x,h(x),\ldots,h^{N}(x)$ are distinct. 
Choose disjoint open neighborhoods $U_{0},\ldots,U_{N}$ of these 
points, and let $U = \bigcap_{j=0}^{N} h^{-j}(U_{j})$, which is an open 
neighborhood of $x$ such that $U,h(U),\ldots,h^{N}(U)$ are pairwise 
disjoint. Then using the $h$-invariance of $\mu$, it follows that 
\[
(N+1) \mu (U) = \sum_{j=0}^{N} \mu (h^{j}(U)) = \mu \left( 
\ts{ \bigcup_{j=0}^{N} h^{j}(U) } \right) \leq \mu (X) = 1,
\]
which gives $\mu (U) < 1/N < \eps$. Since $F$ is thin, we have $F 
\prec U$, and so there exist $M \in \N$, $U_{0},\ldots,U_{M} \subset 
X$ open, and $d(0),\ldots,d(M) \in \Z$ such that $F \subset 
\bigcup_{j=0}^{M} U_{j}$ and such that the sets $h^{d(j)}(U_{j})$ 
are pairwise disjoint subsets of $U$ for $0 \leq j \leq M$. Then 
again using the $h$-invariance of $\mu$, we have
\begin{align*}
\mu (F) \leq \mu \left( \ts{ \bigcup_{j=0}^{M} U_{j} } \right) \leq 
\sum_{j=0}^{M} \mu (U_{j}) &= \sum_{j=0}^{M} \mu ( h^{d(j)} (U_{j}) ) \\ 
&= \mu \left( \ts{ \bigcup_{j=0}^{M} h^{d(j)}(U_{j}) } \right) \leq \mu (U) 
< \eps.
\end{align*}
Since $\eps > 0$ was arbitrary, it follows that $\mu (F) = 0$.
\end{proof}

\begin{lem}\label{ThinUnionThin}
Let $(X,h)$ be as in Notation \ref{MinDynSys}.
\begin{enumerate}
\item If $F_{1},F_{2} \subset X$ are closed and $V_{1},V_{2} \subset 
X$ are open such that $F_{1} \prec V_{1}$, $F_{2} \prec V_{2}$, and 
$V_{1} \cap V_{2} = \varnothing$, then $F_{1} \cup F_{2} \prec V_{1} 
\cup V_{2}$.
\item The union of finitely many thin sets in $X$ is thin.
\end{enumerate}
\end{lem}

\begin{proof}
To prove (1), simply observe that since $V_{1} \cap V_{2} = 
\varnothing$, the union of a pairwise disjoint collection of subsets 
of $V_{1}$ and a pairwise disjoint collection of subsets of $V_{2}$ 
is still pairwise disjoint.

For (2), it is sufficient to prove that the union of two thin sets 
is thin. Let $F_{1}, F_{2} \subset X$ be thin closed sets, and let 
$U \subset X$ be a non-empty open set. Since $h$ is minimal there 
must be distinct points $x_{1}, x_{2} \subset U$. Let $V_{1} \subset 
U$ and $V_{2} \subset U$ be disjoint open neighborhoods of $x_{1}$ 
and $x_{2}$ respectively. Then $F_{1} \prec V_{1}$ and $F_{2} \prec 
V_{2}$, and now part 1 implies that $F_{1} \cup F_{2} \prec V_{1} 
\cup V_{2} \subset U$, which proves that $F_{1} \cup F_{2}$ is thin.
\end{proof}

\begin{lem}\label{ThinWithClosedNbds}
Let $(X,h)$ be as in Notation \ref{MinDynSys}. 
Let $F \subset X$ be a thin closed set, and let $U \subset X$ be
open. Then there exist $M \in \N$, $F_{0},\ldots,F_{M} \subset X$
closed, and $d(0),\ldots,d(M) \in \Z$ such that:
\begin{enumerate}
\item $F \subset \bigcup_{j=0}^{M} F_{j}$;
\item $h^{d(j)}(F_{j}) \subset U$ for $0 \leq j \leq M$;
\item the sets $h^{d(j)}(F_{j})$ are pairwise disjoint for $0 \leq j
\leq M$.
\end{enumerate}
\end{lem}

\begin{proof}
Since $F$ is thin, we have $F \prec U$, and so there exist $M \in 
\N$, $U_{0},\ldots,U_{M} \subset X$ open, and $d(0),\ldots,d(M) \in 
\Z$ such that $F \subset \bigcup_{j=0}^{M} U_{j}$ and the sets
$h^{d(j)}(U_{j})$ are pairwise disjoint subsets of $U$ for $0 \leq
j \leq M$. Now temporarily fix $j \in \set{0,\ldots,M}$. For each
$x \in U_{j}$, let $V_{x}^{(j)}$ be a neighborhood of $x$ such that
$V_{x}^{(j)} \subset \overline{V}_{x}^{(j)} \subset U_{j}$. Then
$\set{V_{x}^{(j)} \colon x \in U_{j}, 0 \leq j \leq M}$ is an open
cover for $F$, hence it contains a finite subcover. For $0 \leq j
\leq M$ let $\mathcal{S}_{j}$ be the (possibly empty) collection of
all sets $V_{x}^{(j)}$ that appear in the finite subcover for $F$,
and set $F_{j} = \bigcup_{V \in \mathcal{S}_{j}} \overline{V}$. Note
that $F_{j} = \varnothing$ if the collection $\mathcal{S}_{j}$ is
empty. Then each $F_{j}$ is closed (being the union of finitely 
many closed sets) and satisfies $F_{j} \subset U_{j}$. It follows 
that the sets $h^{d(j)}(F_{j})$ are pairwise disjoint subsets of $U$ 
for $0 \leq j \leq M$.
\end{proof}

\begin{lem}\label{DistinctIntegers}
Suppose that $d_{0},\ldots,d_{m}$ are $m+1$ distinct integers, 
and that $n_{1},n_{2}$ are distinct integers (but not necessarily 
distinct from the $d_{i}$). Then the set 
\[
\set{d_{i} + n_{j} \colon 0 \leq i \leq m, j = 1,2} 
\]
contains at least $m+2$ distinct integers.
\end{lem}

\begin{proof}
Without loss of generality, suppose that $d_{0} < d_{1} < \cdots < 
d_{m}$ and $n_{1} < n_{2}$. Then we have 
\[
d_{0} + n_{1} < d_{1} + n_{1} < \cdots < d_{m} + n_{1} < d_{m} + 
n_{2},
\]
which provides $m + 2$ distinct integers in the set $\set{d_{i} + 
n_{j} \colon 0 \leq i \leq m, j = 1,2}$.
\end{proof}

\begin{prp}\label{TopSmallIsThin}
Let $(X,h)$ be as in Notation \ref{MinDynSys}. 
If $F \subset X$ is topologically $h$-small, then $F$ is thin.
\end{prp}

\begin{proof}
The proof is by induction on the smallness constant $m$. First 
consider the case where the smallness constant is $m = 1$. Then
given $j,k \in \Z$ with $j \neq k$, we have $h^{j}(F) \cap h^{k}(F)
= \varnothing$. Let $U \subset X$ be open and non-empty, and let 
$V_{0} \subset U$ be open and non-empty with $\overline{V}_{0} 
\subset X$. By Lemma \ref{OpenPositiveMeas}, $\set{ h^{n}(V_{0}) 
\colon n \in \Z}$ is an open cover for $F$, so there exists a finite 
subcover $\set{h^{-d(0)}(V_{0}),\ldots,h^{-d(M)}(V_{0})}$. Set 
$F_{j} = F \cap \overline{h^{-d(j)}(V_{0})}$. Then the sets $h^{d(j)}
(F_{j})$ are closed, disjoint (since $h^{d(j)}(F_{j}) \subset 
h^{d(j)}(F)$ and these sets are disjoint) and satisfy $h^{d(j)}
(F_{j}) \subset \overline{V}_{0} \subset U$. Since $X$ is normal, 
there exist disjoint open sets $W_{0},\ldots,W_{M} \subset X$ such 
that $h^{d(j)}(F_{j}) \subset W_{j}$. Finally, for $0 \leq j \leq M$ 
set $U_{j} = h^{-d(j)} (W_{j} \cap U)$. Then $F \subset 
\bigcup_{j=0}^{M} U_{j}$, and the sets $h^{d(j)}(U_{j})$ are pairwise 
disjoint (being subsets of the $W_{j}$) and contained in $U$.

Now let $m \geq 1$, and suppose that closed sets which are 
topologically $h$-small with smallness constant $m$ are thin. Let $F 
\subset X$ be topologically $h$-small with smallness constant $m+1$. 
For $j,k \in \Z$ with $j \neq k$, define $F_{j,k} = h^{j}(F) \cap 
h^{k}(F)$. We claim that the sets $F_{j,k}$ are topologically 
$h$-small with smallness constant $m$. To see this, let $d_{0},
\ldots,d_{m}$ be $m+1$ distinct integers, and let $j,k \in \Z$ with 
$j \neq k$. By Lemma \ref{DistinctIntegers}, the set $\set{d_{i} + l 
\colon, 0 \leq i \leq m, l = j,k}$ contains at least $m+2$ distinct 
integers. It follows that 
\[
h^{d_{0}}(F_{j,k}) \cap \cdots \cap h^{d_{m}} (F_{j,k}) = 
\bigcap_{i=0}^{m} ( h^{d_{i} + j}(F) \cap h^{d_{i} + k}(F) ) = 
\varnothing,
\]
which proves the claim. Now choose disjoint, non-empty open sets 
$V_{1},V_{2} \subset U$, and choose disjoint, non-empty open sets 
$Z_{1},Z_{2}$ with $\overline{Z}_{1} \subset V_{1}$ and 
$\overline{Z}_{2} \subset V_{2}$. By Lemma \ref{OpenPositiveMeas}, 
the collection $\set{h^{n}(Z_{1}) \colon n \in \Z}$ is an open cover 
for $F$, so it contains a finite subcover $\set{h^{-n_{0}}(Z_{1}),
\ldots,h^{-n_{K}} (Z_{1})}$. Set $T = \set{(j,k) \colon 0 \leq j < k 
\leq K}$ and for each $(j,k) \in T$ define $D_{j,k} = h^{n_{j}}(F) 
\cap h^{n_{k}}(F) \cap \overline{Z}_{1}$, which is a closed subset 
of $F_{n_{j},n_{k}}$. By the earlier claim, $D_{j,k}$ is 
topologically $h$-small with smallness constant $m$, and so it is 
thin by the induction hypothesis. Choose pairwise disjoint open 
sets $S_{j,k} \subset Z_{2}$ for $(j,k) \in T$. Since each $D_{j,k}$ 
is thin, there exist $M(j,k) \in \N$, $U_{j,k,0}^{(0)}, \ldots, 
U_{j,k,M(j,k)}^{(0)} \subset X$ open, and $d_{j,k}(0), \ldots, 
d_{j,k}(M(j,k)) \in \Z$ such that:
\begin{enumerate}
\item $D_{j,k} \subset \bigcup_{i=0}^{M(j,k)} U_{j,k,i}^{(0)}$; 
\item $h^{d_{j,k}(i)} ( U_{j,k,i}^{(0)} ) \subset S_{j,k}$;
\item the sets $h^{d_{j,k}(i)} ( U_{j,k,i}^{(0)} )$ are 
pairwise disjoint for $0 \leq i \leq M(j,k)$.
\end{enumerate}
Now set 
\[
D = \bigcup_{(j,k) \in T} h^{-n_{j}} ( D_{j,k} ) 
\hspace{0.5 in} \textnormal{and} \hspace{0.5 in} 
W_{0} = \bigcup_{(j,k) \in T} h^{-n_{j}} \left( \ts{ \bigcup_{i=0}^{ 
M(j,k) } U_{j,k,i}^{(0)} } \right).
\]
Then $D$ is closed, $W_{0}$ is open, and $D \subset W_{0}$. Choose 
$W \subset X$ open such that $D \subset W \subset \overline{W} 
\subset W_{0}$. For $0 \leq j \leq K$, set $F_{j} = h^{-n_{j}}( 
\overline{Z}_{1} ) \cap (X \setminus W) \cap F$, which is closed. 
Let $x \in F$ and suppose $x \not \in W$. For some $j \in \set{0,
\ldots,K}$, we have $x \in h^{-n_{j}} (Z_{1})$. Then $x \in F$, $x 
\in h^{-n_{j}} (\overline{Z}_{1})$, and $x \in X \setminus W$, so $x 
\in F_{j}$. It follows that $\set{F_{0},\ldots,F_{K},W}$ covers $F$. 
Next suppose that $x \in h^{n_{j}}(F_{j}) \cap h^{n_{k}}(F_{k})$ for 
some $(j,k) \in T$. Then there are $x_{j} \in F_{j}$ and $x_{k} \in 
F_{k}$ such that $h^{n_{j}}(x_{j}) = x = h^{n_{k}}(x_{k})$. Since 
$F_{j}, F_{k} \subset F$ we certainly have $x \in h^{n_{j}}(F) \cap 
h^{n_{k}}(F)$. Moreover, $x_{j} = h^{-n_{j}}(x) \in h^{-n_{j}} ( 
\overline{Z}_{1} )$, which gives $x \in \overline{Z}_{1}$. It 
follows that $x \in D_{j,k}$, and so also $x_{j} = h^{-n_{j}}(x) \in 
h^{-n_{j}}(D_{t_{j,k}}) \subset W$. This implies $x_{j} \not \in 
F_{j}$, a contradiction. Therefore, the sets $h^{n_{j}} (F_{j})$ are 
pairwise disjoint. Since $h^{n_{j}}(F_{j}) \subset 
\overline{Z}_{1}$, they are all subsets of $V_{1}$. Using the 
normality of $X$, choose non-empty pairwise disjoint open sets 
$U_{0}^{(0)}, \ldots, U_{K}^{(0)} \subset X$ such that 
$h^{n_{j}}(F_{j}) \subset U_{j}^{(0)} \subset V_{1}$. For an 
appropriate $M \in \N$, re-index the sets 
\[
\set{ h^{-n_{0}}(U_{0}^{(0)}),\ldots,h^{-n_{K}}(U_{K}^{(0)}) } \cup 
\set{ h^{-n_{j}} ( U_{j,k,i}^{(0)}) \colon (j,k) \in T, 0 \leq 
i \leq M(j,k) }
\]
and
\[
\set{ n_{0}, \ldots, n_{K} } \cup \set{ n_{j} + d_{j,k}(i) 
\colon (j,k) \in T, 0 \leq i \leq M(j,k) }
\]
as $\set{ U_{0}, \ldots, U_{M} }$ and $\set{ d(0), \ldots, d(M) }$ 
respectively. Then $F \subset \bigcup_{i=0}^{M} U_{j}$ and the sets 
$h^{d(j)}(U_{j})$ are pairwise disjoint subsets of $U$ for $0 \leq j 
\leq M$. It follows that $F$ is thin, completing the induction.
\end{proof}

\begin{cor}\label{TopSmallMeasZero}
Let $(X,h)$ be as in Notation \ref{MinDynSys}. Let $F \subset X$ be 
closed and topologically $h$-small. Then $F$ is universally null.
\end{cor}

\begin{proof}
This follows immediately from Proposition \ref{TopSmallIsThin} and 
Lemma \ref{ThinHasMeasZero}.
\end{proof}

The next lemma will play a crucial role in the proof of our main 
result. It provides a strong decomposition property for thin sets.

\begin{lem}\label{LeftoverCoverSpace}
Let $(X,h)$ be as in Notation \ref{MinDynSys}. 
Let $\eps > 0$ be given, and let $F \subset X$ be thin. Then for any
non-empty open set $U \subset X$ there exist $M \in \N$, closed sets
$F_{j} \subset X$ for $0 \leq j \leq M$, open sets $T_{j},V_{j},
W_{j} \subset X$ for $0 \leq j \leq M$, continuous functions $f_{0},
\ldots, f_{M} \colon X \to [0,1]$, and $d(0),\ldots,d(M) \in \Z$ 
such that:
\begin{enumerate}
\item $F \subset \bigcup_{j=0}^{M} F_{j}$;
\item $h^{-d(j)}(F_{j}) \subset T_{j} \subset \overline{T}_{j}
\subset V_{j} \subset \overline{V}_{j} \subset W_{j} \subset U$
for $0 \leq j \leq M$;
\item $\sum_{j=0}^{M} f_{j} = 1$ on $\bigcup_{j=0}^{M}
h^{d(j)}( \overline{V}_{j} )$;
\item $\supp (f_{j} \circ h^{-d(j)}) \subset W_{j}$ for $0 \leq j
\leq M$;
\item the sets $W_{j}$ are pairwise disjoint and $\sum_{j=0}^{M} \mu
(W_{j}) < \eps$ for all $\mu \in M_{h}(X)$.
\end{enumerate}
\end{lem}

\begin{proof}
Since $U$ is open and non-empty, Corollary 
\ref{MeasApproxZeroBdy} implies there is a non-empty open set 
$E \subset U$ with $\mu (E) < \eps$ for all $\mu \in M_{h}(X)$. 
Since $F$ is thin, we can apply Lemma \ref{ThinWithClosedNbds} to
$F$ and $E$, which implies there exist $M \in \N$, $F_{0},\ldots,
F_{M} \subset X$ closed, and $k(0),\ldots,k(M) \in \Z$ such that $F 
\subset \bigcup_{j=0}^{M} F_{j}$ and such that the sets $h^{k(j)}
(F_{j})$ are pairwise disjoint subsets of $E$. For $0 \leq j \leq 
M$, we set $d(j) = - k(j)$. Since $X$ is normal, we may choose for 
$0 \leq j \leq M$ open sets $W_{j}$ with $F_{j} \subset W_{j} 
\subset E$ such that the $W_{j}$ are pairwise disjoint. Now we can 
use the compactness of $X$ to obtain open sets $T_{j},V_{j} \subset 
X$ such that
\[
h^{-d(j)} (F_{j}) \subset T_{j} \subset \overline{T}_{j} \subset
V_{j} \subset \overline{V}_{j} \subset W_{j}.
\]
For $0 \leq j \leq M$ choose continuous functions $g_{j} \colon X
\to [0,1]$ such that $g_{j} = 1$ on $h^{d(j)} (\overline{V}_{j})$
and $\supp(g_{j}) \subset h^{d(j)} (W_{j})$. Then $\sum_{j=0}^{M}
g_{j}(x) \geq 1$ for all $x \in \bigcup_{j=0}^{M} h^{d(j)}
(\overline{V}_{j})$. By the continuity of the $g_{j}$, there is an 
open set $Q \subset X$ such that $\bigcup_{j=0}^{M} h^{d(j)} 
(\overline{V}_{j}) \subset Q$ and $\sum_{j=0}^{M} g_{j}(x) \geq 
\ts{\frac{1}{2}}$ for all $x \in Q$. Choose a continuous function 
$f \colon X \to [0,1]$ such that $f = 1$ on $\bigcup_{j=0}^{M} 
h^{d(j)}(\overline{V}_{j})$ and $\supp (f) \subset Q$. Now, for 
$0 \leq j \leq M$, define continuous functions $f_{j} \colon X \to 
[0,1]$ by 
\[ 
f_{j}(x) = \begin{cases} f(x) g_{j}(x) \left( \sum_{i=0}^{M} 
g_{i}(x) \right)^{-1} & \textnormal{if} \; x \in Q \\ 0 & 
\textnormal{if} \; x \not \in Q \end{cases} 
\]
Then for any $x \in \bigcup_{j=0}^{M} 
h^{d(j)}(\overline{V}_{j})$, we have 
\[
\sum_{j=0}^{M} f_{j}(x) = \sum_{j=0}^{M} f(x) g_{j}(x) \ts{\left( 
\sum_{i=0}^{M} g_{i}(x) \right)^{-1}}= \ts{\left( \sum_{i=0}^{M} 
g_{i}(x) \right)^{-1}} \sum_{j=0}^{M} g_{j}(x) = 1.
\]
In particular, $\sum_{j=0}^{M} f_{j} = 1$ on $\bigcup_{j=0}^{M} 
h^{d(j)}(T_{j})$. Moreover, $\supp (f_{j}) = \supp (g_{j}) \subset 
h^{d(j)}(W_{j})$, which implies that $\supp (f_{j} \circ h^{-d(j)}) 
= \supp (g_{j} \circ h^{-d(j)}) \subset W_{j}$. Finally, as the 
$W_{j}$ are pairwise disjoint subsets of $E$ for $0 \leq j \leq M$, 
it follows that for any $\mu \in M_{h}(X)$, we have
\[
\sum_{j=0}^{M} \mu (W_{j}) = \mu \left( \ts{ \bigcup_{j=0}^{M} W_{j} }
\right) \leq \mu (E) < \eps,
\]
which completes the proof.
\end{proof}

\section{Small Boundary Properties}\label{SectionSBPs}

In \cite{GK}, Giol and Kerr established that classification 
of transformation group $C^{*}$-algebras by their Elliott invariants 
is intractable for minimal dynamical systems having positive mean 
dimension (in the sense of \cite{LiWe}). For this reason, we only 
wish to consider systems with strong enough regularity properties 
to imply the system has mean dimension zero. For minimal dynamical 
systems, the following definition is equivalent to mean dimension 
zero by Theorem 6.2 of \cite{Lind} and Theorem 5.4 of \cite{LiWe}.

\begin{dfn}\label{SBP}
Let $(X,h)$ be as in Notation \ref{MinDynSys}. We say $(X,h)$ has 
the {\emph{small boundary property}} if for every $x \in X$ and every 
open neighborhood $U$ of $x$, there is an open neighborhood $V$ 
of $x$ such that $V \subset U$ and $\del V$ is universally null.
\end{dfn}

As noted earlier, being universally null might be seen as a weak form 
of smallness. We thus introduce a new definition where this is replaced 
with one of our stronger topological conditions.

\begin{dfn}\label{TSBP}
We say $(X,h)$ has the {\emph{topological small boundary property}} 
if whenever $F,K \subset X$ are disjoint compact sets, then there exist 
open sets $U,V \subset X$ such that $F \subset U$, $K \subset V$, 
$\overline{U} \cap \overline{V} = \varnothing$, and $\del U$ is 
topologically $h$-small.
\end{dfn}

Definition \ref{TSBP} first appeared in \cite{PhTopSmall} as a purely 
topological analogue of a transversality property for manifolds.

\begin{lem}\label{TSBPEquiv}
Let $(X,h)$ be as in Notation \ref{MinDynSys}.
\begin{enumerate}
\item\label{TopSmallSingle} If $(X,h)$ has the topological small 
boundary property, then whenever $F \subset X$ is a compact set, 
there is an open set $U \subset X$ such that $F \subset U$ and 
$\del U$ is topologically $h$-small.
\item\label{TopSmallPoint} $(X,h)$ has the topological small 
boundary property if and only if for any $x \in X$ and any open 
neighborhood $U$ of $x$, there is an open neighborhood $V$ of 
$x$ such that $V \subset \overline{V} \subset U$ and $\del V$ is 
topologically $h$-small.
\end{enumerate}
\end{lem}

\begin{proof}
\begin{enumerate}
\item This follows immediately by applying the condition of the 
topological small boundary property to the compact sets $F$ and 
$K = \varnothing$, then disregarding the open set $V$ obtained.
\item If $(X,h)$ has the topological small boundary property, $x \in 
X$, and $U$ is an open neighborhood of $x$, then by the topological 
small boundary property, there exist open sets $V, W \subset X$ such 
that $\set{x} \subset V$, $X \setminus U \subset W$, $\overline{V} 
\cap \overline{W} = \varnothing$, and $\del V$ is topologically 
$h$-small. Since $X \setminus U \subset \overline{W}$, it follows that 
$\overline{V} \subset U$, as required. For the converse, let $F, K 
\subset X$ be compact. Since $X$ is normal, there exist open sets 
$E, W \subset X$ such that $F \subset E$, $K \subset W$, and $E 
\cap W = \varnothing$. By assumption, for each $x \in F$ there is an 
open set $U_{x} \subset X$ such that $x \in U_{x} \subset 
\overline{U}_{x} \subset E$ and $\del U_{x}$ is topologically 
$h$-small. Since $F$ is compact the open cover $\set{U_{x} \colon 
x \in X}$ contains a finite subcover $\set{U_{x_{1}},\ldots,U_{x_{n}}}$ 
for $F$. Set $U = \bigcup_{j=1}^{n} U_{x_{j}}$. Then $F \subset U 
\subset \overline{U} \subset W$, and $\del U$ is topologically 
$h$-small, since $\bigcup_{j=1}^{n} \del U_{x_{j}}$ is topologically 
$h$-small by Lemma \ref{TopSmallProps}(\ref{TopSmallUnion}) and 
$\del U \subset \bigcup_{j=1}^{n} \del U_{x_{j}}$. Finally, the 
compactness and local compactness of $K$ imply there is an open 
set $V$ with $K \subset V \subset \overline{V} \subset W$. Then 
$\overline{U} \cap \overline{V} = \bigcup_{j=1}^{n} \overline{U}_{x_{j}} 
\cap \overline{V} \subset E \cap W = \varnothing$. It follows that $(X,h)$ 
has the topological small boundary property.
\end{enumerate}
\end{proof}

\begin{cor}\label{TSBPImpliesSBP}
Let $(X,h)$ be as in Notation \ref{MinDynSys}. If $(X,h)$ has the 
topological small boundary property, then $(X,h)$ has the small 
boundary property. Consequently, $(X,h)$ has mean dimension 
zero.
\end{cor}

\begin{proof}
Since Corollary \ref{TopSmallMeasZero} implies that topologically 
$h$-small sets are universally null, this follows immediately from 
Proposition \ref{TSBPEquiv}(\ref{TopSmallPoint}).
\end{proof}

\begin{prp}\label{FinDimTSBP}
Let $(X,h)$ be as in Notation \ref{MinDynSys}. Assume in addition 
that $\dim (X) = d < \infty$. Then $(X,h)$ has the topological small 
boundary property.
\end{prp}

\begin{proof}
Let $x \in X$ and let $U \subset X$ be an open neighborhood of 
$X$. Let $\eps = \dist (x, X \setminus U) > 0$. We apply Lemma 
3.7 of \cite{Kul} with $i > 1/\eps$ there to obtain closed sets 
$F_{1},\ldots,F_{n} \subset X$ which cover $X$ such that for each 
$1 \leq j \leq n$, $\sint (F_{j}) \neq \varnothing$, $\diam  (F_{j}) < 
1/i$, $\overline{\sint (F_{j})} = F_{j}$, and such that, whenever 
$k(0),\ldots,k(d) \in \Z$ are distinct, then 
\[
h^{k(0)}(\del F_{j}) \cap h^{k(1)}(\del F_{j}) \cap \cdots \cap h^{k(d)} 
(\del F_{j}) = \varnothing.
\]
It follows that each $\del F_{j}$ is topologically $h$-small. Now, 
there is some $M \in \set{0,\ldots,n}$ such that $x \in F_{M}$. Then 
$\diam (F_{M}) < 1/i < \eps$ implies that, for any $y \in F_{M}$, we 
have $d(x,y) < \dist (x, X \setminus U)$. It follows that $y \in U$, and 
thus $F_{M} \subset U$. Then $V = \sint (F_{j})$ is an open 
neighborhood of $x$ such that $V \subset \overline{V} \subset U$ 
and $\del V$ is topologically $h$-small. Now Lemma 
\ref{TSBPEquiv}(\ref{TopSmallPoint}) implies that $(X,h)$ has 
the topological small boundary property.
\end{proof}

\begin{prp}\label{ProductTSBP}
Let $(X_{n},h_{n})_{n=1}^{\infty}$ be a sequence of dynamical 
systems consisting of infinite compact metrizable spaces $X_{n}$ 
and homeomorphisms $h_{n} \colon X_{n} \to X_{n}$. Let $X = 
\prod_{n \in \N} X_{n}$ and $h = \prod_{n \in \N} h_{n}$. Assume 
that each system $(X_{n},h_{n})$ has the topological small 
boundary property. Then the dynamical system $(X,h)$ has the 
topological small boundary property. If each system 
$(X_{n},h_{n})$ is minimal, then so is $(X,h)$.
\end{prp}

\begin{proof}
Let $x \in X$ and let $U \subset X$ be a neighborhood of $x$. 
For each $n \in \N$, let $\pi_{n} \colon X \to X_{n}$ denote the 
projection map. We may write $U = \prod_{n \in \N} U_{n}$, 
where each $U_{n} \subset X_{n}$ is open and $U_{n} = X_{n}$ 
for all but finitely many choices of $n$. Without loss of generality 
we may assume that there is an $N \in \N$ such that $U_{n} \neq 
X_{n}$ for $1 \leq n \leq N$ and $U_{n} = X_{n}$ for $n \geq N$.
For each $1 \leq n \leq N$, $U_{n} = \pi_{n}(U)$ is an open 
neighborhood of $x$ in $X_{n}$, and the topological small 
boundary property for $(X_{n},h_{n})$ implies that there exists 
an open neighborhood $V_{n} \subset X_{n}$ of $\pi_{n}(x)$ 
such that $V_{n} \subset \overline{V}_{n} \subset U_{n}$ with 
$\del V_{n}$ topologically $h$-small. Let $m_{n}$ denote the 
topological smallness constant for $\del V_{n}$. Define an 
open neighborhood $V \subset X$ of $x$ by 
\[
V = \left( \ts{ \prod_{n \leq N} V_{n} } \right) \times \left( 
\ts{ \prod_{n > N} X_{n} } \right).
\]
Then $x \in V \subset \overline{V} \subset U$ and $\del V$ is 
topologically $h$-small, with topological smallness constant 
$m = \max \set{m_{n} \colon 1 \leq n \leq N}$.
\end{proof}

Taylor Hines \cite{Hines} has constructed an example of a 
minimal infinite-dimensional dynamical system which is of the 
form considered in Proposition \ref{ProductTSBP}. Since that 
work has not yet been published, we describe the system here.

\begin{cor}\label{HinesExample}
There is a minimal dynamical system $(X,h)$ with $\dim (X) 
= \infty$ that has the topological small boundary property.
\end{cor}

\begin{proof}
Let $\set{ \te_{n} \colon n \in \N}$ be a collection of irrational 
numbers $\te_{n}$ that is rationally independent. For $n \in 
\N$ let $h_{\te_{n}} \colon S^{1} \to S^{1}$ be given by 
$h_{\te_{n}} (\zt) = e^{2 \pi i \te_{n}} \zt$. For $n \in \N$, let 
$X_{n} = \prod_{j=1}^{n} S^{1}$ and $h_{n} = \prod_{j=1}^{n} 
h_{\te_{j}}$. Then each system $(X_{n},h_{n})$ is minimal by 
the rational independence of the irrational numbers $\te_{j}$. 
Define $(X,h) = \invlim (X_{n},h_{n})$, where the connecting 
maps $X_{n+1} \to X_{n}$ are the obvious projection maps. 
Then $X = \prod_{n \in \N} S^{1}$ and $h = \prod_{n \in \N} 
h_{\te_{n}}$. Each $X_{n}$ is finite-dimensional, and so 
$(X_{n},h_{n})$ has the topological small boundary property 
by Proposition \ref{FinDimTSBP}. Now the result follows by 
Proposition \ref{ProductTSBP}.
\end{proof}

\begin{lem}\label{TopSmallInnerApprox}
Suppose that $(X,h)$ has the topological small boundary 
property, and let $\eps > 0$ be given. Then for any closed 
set $F \subset X$ and any open set $U \subset X$ with $F 
\subset U$, there exist $K \subset X$ closed and $V \subset 
X$ open such that $F \subset K \subset V \subset \overline{V} 
\subset U$ and such that $\mu (V \setminus K) < \eps$ for 
all $\mu \in M_{h}(X)$.
\end{lem}

\begin{proof}
Let $\eps$, $F$, and $U$ be as given in the statement of the 
Lemma. The topological small boundary property implies 
there are open sets $W,T \subset X$ such that $F \subset W$, 
$X \setminus U \subset T$, $\overline{W} \cap \overline{T} = 
\varnothing$, and $\del W$ is topologically $h$-small. The 
compactness of $F$ and $\overline{W}$ imply that 
\[
\dt = \min \set{ \dist (F, X \setminus W), \dist (\overline{W}, X \setminus 
U) } > 0.
\]
For $n \geq 1$ define $E_{n} = \set{x \in X \colon \dist (x, \del W) < 
\dt/2^{n}}$. An argument entirely analogous to that given in the proof 
of Lemma \ref{ArbSmallOpenSet} implies that there is an $N \in \N$ 
such that $\mu (E_{N}) < \eps$ for all $\mu \in M_{h}(X)$. Now set 
$K = \overline{W} \setminus E_{N}$ and $V = W \cup E_{N}$. If $x 
\in F$ then $\dist (x, \del W) = \dist (x, X \setminus W) \geq \dt > 
\dt/2^{N}$, and so $x \not \in E_{N}$, which implies that $F \subset 
K$. Since $\del W \subset E_{N}$ we have $V = W \cup E_{N} = 
\overline{W} \cup E_{N}$, and so $K \subset \overline{W} \subset V$. 
Let $x \in \overline{E}_{N}$. Then 
\[
\dist (x,X \setminus U) = \dist (x,\del U) \geq \dist (\overline{W}, 
\del U) - \dist (x, \del W) > \dt - \dt/2^{N} > 0,
\]
which implies that $x \in U$. This shows that $\overline{E}_{N} 
\subset U$, and as $\overline{W} \subset U$, it follows that 
$\overline{V} = \overline{W} \cup \overline{E}_{N} \subset U$. Finally, 
$V \setminus K = E_{N}$ by construction, and so we have $\mu 
(V \setminus K) < \eps$ for all $ \mu \in M_{h}(X)$.
\end{proof}

\begin{prp}\label{TopSmallTopSmallInnerApprox}
Suppose that $(X,h)$ has the topological small boundary property, 
and let $\eps > 0$ be given. Then for any closed set $F \subset X$ 
and any open set $U \subset X$ with $F \subset U$, there is an 
closed set $K \subset X$ and an open set $V \subset X$ such that 
\begin{enumerate}
\item $F \subset \sint (K) \subset K \subset V \subset \overline{V} 
\subset U$;
\item $\del K$, $\del U$ are topologically $h$-small;
\item $\mu (V \setminus K) < \eps$ for all $\mu \in M_{h}(X)$.
\end{enumerate}
\end{prp}

\begin{proof}
By Lemma \ref{TopSmallInnerApprox}, there is a closed set $K_{0} 
\subset X$ and an open set $V_{0} \subset X$ such that $F \subset 
K_{0} \subset V_{0} \subset \overline{V}_{0} \subset U$ and such 
that $\mu (V_{0} \setminus K_{0}) < \eps$ for all $\mu \in M_{h}(X)$. 
The topological small boundary property implies there exist open 
sets $E_{0}, E_{1} \subset X$ such that $K_{0} \subset E_{0}$, 
$X \setminus V_{0} \subset E_{1}$, $\overline{E}_{0} \cap 
\overline{E}_{1} = \varnothing$, $\del E_{0}$ and is topologically 
$h$-small. We also immediately obtain that $\overline{E}_{0} 
\subset V_{0}$. Using the topological small boundary property 
again, we obtain open sets $W_{0}, W_{1} \subset X$ such that 
$\overline{E}_{0} \subset W_{0}$, $\overline{E}_{1} \subset W_{1}$, 
$\overline{W}_{0} \cap \overline{W}_{1} = \varnothing$, and $\del 
W_{1}$ is topologically $h$-small. Set $K = \overline{E}_{0}$ and 
$V = X \setminus \overline{W}_{1}$. 

To see that (1) is satisfied, observe first that $F \subset K_{0} \subset 
E_{0} = \sint (K) \subset K$. Next, $\overline{W}_{0} \cap 
\overline{W}_{1} = \varnothing$ implies that $\overline{W}_{0} 
\subset V$, and so $K = \overline{E}_{0} \subset W_{0} \subset 
\overline{W}_{0} \subset V$. Finally, if $x \in V$ then $x \not \in 
\overline{W}_{1}$, and so $x \not \in X \setminus V_{0}$, which 
implies that $x \in V_{0}$. This gives $V \subset V_{0}$, and so 
$\overline{V} \subset \overline{V}_{0} \subset U$. Together, these 
three observations give (1). For (2), we observe that by construction,
 $\del K = \del \overline{E}_{0} = \del E_{0}$ and $\del V = \del (X 
 \setminus \overline{W}_{1}) = \del W_{1}$. Therefore, both $\del K$ 
 and $\del V$ are topologically $h$-small. Finally, the containments 
 $K_{0} \subset K \subset V \subset V_{0}$ imply that $\mu (V 
 \setminus K) \leq \mu (V_{0} \setminus K_{0}) < \eps$ for all $\mu \in 
 M_{h}(X)$. This gives (3) and completes the proof.
\end{proof}

\begin{cor}\label{TopSmallRegularity}
Suppose that $(X,h)$ has the topological small boundary property.
\begin{enumerate}
\item\label{TopSmallInnerRegular} For any $\eps > 0$ and any 
open set $U \subset X$ with $\del U$ universally null, there is an 
open set $V \subset X$ such that $\overline{V} \subset U$, $\del V$ 
is topologically $h$-small, and $\mu (U \setminus \overline{V}) < 
\eps$ for all $\mu \in M_{h}(X)$.
\item\label{TopSmallCptOuterRegular} For any $\eps > 0$, any 
open set $U \subset X$, and any closed set $F \subset U$ with 
$\del F$ universally null, there is an open set $V \subset X$ such 
that $F \subset V \subset \overline{V} \subset U$, $\del V$ is 
topologically $h$-small, and $\mu (V \setminus F) < \eps$ for all 
$\mu \in M_{h}(X)$.
\item\label{TopSmallOuterRegular} For any $\eps > 0$ and any 
closed set $F \subset X$ with $\del F$ universally null, there is an 
open set $V \subset X$ such that $F \subset V$, $\del V$ is 
topologically $h$-small, and $\mu (V \setminus F) < \eps$ for all 
$\mu \in M_{h}(X)$.
\end{enumerate}
\end{cor}

\begin{proof}
\begin{enumerate}
\item By Corollary \ref{MeasApproxZeroBdy}(\ref{ClosedInsideOpen}), 
there is a closed set $K \subset U$ such that $\mu (U \setminus K) < 
\eps$ for all $\mu \in M_{h}(X)$. Applying Proposition 
\ref{TopSmallTopSmallInnerApprox} to $K$ and $U$, we obtain an 
open set $V \subset X$ such that $K \subset V \subset \overline{V} 
\subset U$ and $\del V$ is topologically $h$-small. Then $K \subset 
\overline{V}$ implies that $\mu (U \setminus \overline{V}) \leq \mu 
(U \setminus K) < \eps$ for all $\mu \in M_{h}(X)$. 
\item By Corollary \ref{MeasApproxZeroBdy}(\ref{OpenOutsideClosed}), 
there is an open set $E \subset X$ such that $F \subset E$ and $\mu 
(E \setminus F) < \eps$ for all $\mu \in M_{h}(X)$. Applying Proposition 
\ref{TopSmallTopSmallInnerApprox} to $F$ and $U \cap E$ (which is 
an open set containing $F$), we obtain an open set $V \subset X$ 
such that $F \subset V \subset \overline{V} \subset U \cap E \subset U$ 
and $\del V$ is topologically $h$-small. Then $V \subset U \cap E 
\subset E$ implies that $\mu (V \setminus F) \leq \mu (E \setminus F) < 
\eps$ for all $\mu \in M_{h}(X)$.
\item This follows immediately from (\ref{TopSmallCptOuterRegular}) 
by taking $U = X$.
\end{enumerate}
\end{proof}

\section{The Dynamic Comparison Property}\label{SectionDCP}

In Lemma 2.5 of \cite{GW}, it is shown that if $(X,h)$ is a Cantor 
minimal system and $A,B \subset X$ are compact-open sets with 
$\mu (A) < \mu (B)$ for all $\mu \in M_{h}(X)$, then in fact there is 
a decomposition $\bigcup A_{i}$ of $A$ such that the sets $A_{i}$ 
can be translated disjointly into $B$. The following definition is an 
attempt to formulate a similar condition for more general minimal 
dynamical systems. In general, we cannot expect non-trivial 
compact-open sets to exist, and so we must instead use open and 
closed sets with some assumed good behavior of their boundaries.

\begin{dfn}\label{DCP}
Let $(X,h)$ be as in Notation \ref{MinDynSys}. 
We say $(X,h)$ has the {\emph{dynamic comparison property}} if 
whenever $U \subset X$ is open and $C \subset X$ is closed with 
$\del C$, $\del U$ universally null and $\mu (C) < \mu (U)$ for every 
$\mu \in M_{h}(X)$, then there are $M \in \N$, continuous functions 
$f_{j} \colon X \to [0,1]$ for $0 \leq j \leq M$, and $d(0),\ldots,d(M) \in 
\Z$ such that $\sum_{j=0}^{M} f_{j} = 1$ on $C$, and such that the 
sets $\supp (f_{j} \circ h^{-d(j)})$ are pairwise disjoint subsets of $U$ 
for $0 \leq j \leq M$.
\end{dfn}

The next lemma gives a condition that implies the dynamic 
comparison property holds for systems with the topological small 
boundary property, and is easier to verify because of the assumed 
additional structure for the closed and open sets involved.

\begin{lem}\label{DCPSimplification}
Let $(X,h)$ be as in Notation \ref{MinDynSys} and assume that 
$(X,h)$ has the topological small boundary property. 
Suppose that $X$ has the property that if whenever $F \subset X$
is closed with $\sint (F) \neq \varnothing$ and $\del F$ 
topologically $h$-small, $E \subset X$ is open, and there exists an 
open set $E_{0} \subset E$ with $\overline{E}_{0} \subset E$,
$\overline{E_{0}} \cap F = \varnothing$, $\del E_{0}$ topologically 
$h$-small, and $\mu (F) < \mu (E_{0})$ for every $\mu \in M_{h}(X)$, 
then there exist $M \in \N$, continuous functions $f_{j} \colon X \to 
[0,1]$ for $0 \leq j \leq M$, and $d(0),\ldots,d(M) \in \Z$ such that 
$\sum_{j=0}^{M}f_{j} = 1$ on $F$, and such that the sets $\supp 
(f_{j} \circ h^{-d(j)})$ are pairwise disjoint subsets of $E$ for $0 \leq 
j \leq M$. Then $(X,h)$ has the dynamic comparison property.
\end{lem}

\begin{proof}
Let $U \subset X$ be open and let $C \subset X$ be closed with $\mu 
(C) < \mu (U)$ for every $\mu \in M_{h}(X)$. Set $\dt = \inf_{\mu \in 
M_{h}(X)} [ \mu (U) - \mu (C) ] > 0$. By Corollary 
\ref{TopSmallRegularity}(\ref{TopSmallInnerRegular}), there is an 
open set $U_{0} \subset U$ with $\overline{U}_{0} \subset U$, $\del 
U_{0}$ topologically $h$-small, and $\mu (U \setminus 
\overline{U}_{0}) < \dt/3$ for all $\mu \in M_{h}(X)$. For any $\mu \in 
M_{h}(X)$, we have 
\[
\mu (U_{0}) - \mu (C) = [\mu (U) - \mu (C)] - [\mu (U) - \mu (U_{0}) ] 
< \dt - \dt/3 = 2 \dt/3 > 0.
\]
First suppose that $C \subset \overline{U}_{0}$. Set $M = 0$ and 
$d(0) = 0$, and choose a continuous function $f_{0} \colon X \to [0,1]$ 
such that $f_{0} = 1$ on $\overline{U}_{0}$ and $\supp (f_{0}) \subset 
U$. Then $\sum_{j=0}^{M} f_{j} = f_{0} = 1$ on $C$, and $\supp (f_{0} 
\circ h^{-d(0)}) = \supp (f_{0}) \subset U$ as required. 

So we may assume that $C \cap (X \setminus \overline{U}_{0}) \neq 
\varnothing$. By Corollary 
\ref{TopSmallRegularity}(\ref{TopSmallInnerRegular}), there is an open 
set $V \subset U_{0}$ such that $\overline{V} \subset U_{0}$, $\del V$ 
is topologically $h$-small, and $\mu (U_{0} \setminus \overline{V}) < 
\dt/3$ for all $\mu \in M_{h}(X)$. For any $\mu \in M_{h}(X)$, we have 
\[
\mu (V) - \mu (C) = [\mu (U_{0}) - \mu (C)] - [\mu (U_{0}) - \mu (V)] > 
2\dt/3 - \dt/3 = \dt/3 > 0.
\]
Now, set $\eps = \inf_{\mu \in M_{h}(X)} [ \mu (V) - \mu (C) ] > 0$. 
Applying Corollary 
\ref{TopSmallRegularity}(\ref{TopSmallCptOuterRegular}) three 
times, we obtain open sets $G_{0},G_{1}, G_{2} \subset X$ such that
\[
C \cap \overline{V} \subset G_{0} \subset \overline{G}_{0}
\subset G_{1} \subset \overline{G}_{1} \subset G_{2} \subset
\overline{G}_{2} \subset U_{0},
\]
with $\del G_{i}$ topologically $h$-small for $i = 0,1,2$ (and so also 
$\del G_{i}$ universally null for $i = 0,1,2$ by Corollary 
\ref{TopSmallMeasZero}), $\mu (G_{0} \setminus C \cap \overline{V}) 
< \eps/4$, $\mu (G_{1} \setminus \overline{G}_{0}) < \eps/4$, and 
$\mu (G_{2} \setminus \overline{G}_{1}) < \eps$ for all $\mu \in 
M_{h}(X)$. Set $F_{0} = C \setminus G_{0}$, $E_{0} = U \setminus 
\overline{G}_{1}$, and $E_{1} = V \setminus \overline{G}_{2}$. Then:
\begin{enumerate}
\item $F_{0}$ is closed and non-empty, since $C \cap (X \setminus 
\overline{U}_{0}) \neq \varnothing$ implies $C \cap (X \setminus 
G_{0}) \neq \varnothing$; 
\item $E_{1}$ and $E_{0}$ are both open and non-empty, and 
by construction we have $E_{1} \subset \overline{E}_{1} \subset 
E_{0}$; 
\item $\overline{E}_{1} \cap F_{0} = \varnothing$; 
\item $\del F_{0}$, $\del E_{0}$, and $\del E_{1}$ are universally 
null, being subsets of the universally null sets $\del C \cup \del 
G_{0}$, $\del U \cup \del G_{1}$, and $\del V \cup \del G_{2}$ 
respectively;
\item Observing $C \cap \overline{V} \subset C \cap G_{0}$, for 
every $\mu \in M_{h}(X)$ we have    
\begin{align*}
\mu (E_{1}) - \mu (F_{0}) &= \mu (V \setminus \overline{G}_{2}) - \mu
(C \setminus G_{0}) \\
&= \mu (V) - \mu (V \cap \overline{G}_{2}) - [ \mu(C) - \mu (C \cap
G_{0}) ] \\
&\geq [ \mu(V) - \mu(C) ] + \mu (C \cap G_{0}) \\
&\hspace{0.52 in} - [\mu (C \cap V) + \mu (G_{2} \setminus 
\overline{G}_{1}) + \mu (G_{1} \setminus \overline{G}_{0}) + \mu 
(G_{0} \setminus (C \cap \overline{V})) ] \\
&\geq \eps - [ \mu (G_{2} \setminus \overline{G}_{1}) + \mu (G_{1}
\setminus \overline{G}_{0}) + \mu (G_{0} \setminus C \cap
\overline{V}) ] \\
&> \eps - 3\eps/4 \\
&= \eps/4 \\
&> 0.
\end{align*}
\end{enumerate}
Now Corollary \ref{TopSmallRegularity}(\ref{TopSmallInnerRegular}) 
gives an open set $E \subset E_{1}$ such that $\overline{E} \subset 
E_{1}$, $\del E$ is topologically $h$-small, and $\mu (E_{1} 
\setminus \overline{E}) < \eps/16$ for all $\mu \in M_{h}(X)$. 
Since $F_{0}$ is disjoint from $\overline{E}$, there is an open 
set $W_{0} \subset X$ such that $F_{0} \subset W_{0}$ and $W_{0} 
\cap \overline{E} = \varnothing$. Corollary 
\ref{TopSmallRegularity}(\ref{TopSmallCptOuterRegular}) then 
implies that there is an open set $W \subset X$ such that $F_{0} 
\subset W \subset \overline{W} \subset W_{0}$, $\del W$ is 
topologically $h$-small, and $\mu (W \setminus F_{0}) < \eps/16$ 
for all $\mu \in M_{h}(X)$. Now set $F = \overline{W}$, 
which satisfies $\sint (F) \neq \varnothing$, $\del F$ topologically 
$h$-small (which in particular gives $\mu (F) = \mu (W)$), and 
$\overline{E} \cap F = \varnothing$. For any $\mu \in M_{h}(X)$, 
we have 
\begin{align*}
\mu (E) - \mu (F) &= \mu (E) - \mu (W) \\ 
&> [\mu (E_{1}) - \eps/16] - [\mu (F_{0}) + \eps/16] \\
&= [\mu (E_{1}) - \mu (F_{0})] - \eps/8 \\ 
&> \eps/8 \\ 
&> 0.
\end{align*}
It follows that the sets $F$ and $E$ satisfy the conditions for the 
property given in the statement of the Lemma. Therefore, there exist 
$M \in \N$, continuous functions $f_{0},\ldots,f_{M} \colon X \to [0,1]$, 
and $d(0),\ldots,d(M) \in \Z$ such that $\sum_{j=0}^{M} f_{j} = 1$ 
on $F$, and such that the sets $\supp (f_{j} \circ h^{-d(j)})$ are 
pairwise disjoint subsets of $E$ for $0 \leq j \leq M$. Choose a 
continuous function $f_{M+1} \colon X \to [0,1]$ such that $f_{M+1} = 
1$ on $\overline{G}_{1}$ and $\supp (f_{M+1}) \subset G_{2}$, and 
set $d(M+1) = 0$. Now for any $x \in C$, either $x \in F_{0}$ or $x 
\in G_{0} \cap C$. If $x \in F_{0}$ then in particular $x \in F$, 
and so $\sum_{j=0}^{M+1} f_{j}(x) \geq \sum_{j=0}^{M} f_{j}(x) = 
1$. If $x \in G_{0} \cap C$ then in particular $x \in 
\overline{G}_{1}$, and so $\sum_{j=0}^{M+1} f_{j}(x) \geq
f_{M+1}(x) = 1$. It follows that $\sum_{j=0}^{M+1} f_{j}(x) \geq 1$ 
for all $x \in C$. From the continuity of the $f_{j}$, there 
is an open set $S \subset X$ such that $C \subset S$ and 
$\sum_{j=0}^{M+1} f_{j}(x) \geq \ts{\frac{1}{2}}$ for all $x \in S$. 
Choose a continuous function $f \colon X \to [0,1]$ such that $f = 
1$ on $C$ and $\supp (f) \subset S$. For $0 \leq j \leq M+1$, define 
a continuous function $g_{j} \colon X \to [0,1]$ by 
\[
g_{j}(x) = \begin{cases} f(x) f_{j}(x) \left( \sum_{i=0}^{M+1} 
f_{i}(x) \right)^{-1} & \textnormal{if} \; x \in S \\ 0 & 
\textnormal{if} \; x \not \in S. \end{cases}
\]
Then for any $x \in C$, we have
\[ 
\sum_{j=0}^{M+1} g_{j}(x) = \left( \sum_{i=0}^{M+1} f_{i}(x) 
\right)^{-1} \sum_{j=0}^{M+1} f(x) f_{j}(x) = \left( \sum_{i=0}^{M+1} 
f_{i}(x) \right)^{-1} \sum_{j=0}^{M+1} f_{j}(x) = 1.
\] 
Moreover, $g_{j}(x) = 0$ for any $x \in X$ where $f_{j}(x) = 0$, 
which implies that $\supp (g_{j}) \subset \supp (f_{j})$. It follows 
that $\supp (g_{j} \circ h^{-d(j)}) \subset \supp (f_{j} \circ 
h^{-d(j)})$ for $0 \leq j \leq M+1$. This immediately gives pairwise 
disjointness of the sets $\supp (g_{j} \circ h^{-d(j)})$ for $0 \leq 
j \leq M$, since the sets $\supp (f_{j} \circ h^{-d(j)})$ are 
pairwise disjoint for $0 \leq j \leq M$. Further, all of these sets 
are contained in $U$ as $E \subset U$. Finally, $\supp (g_{M+1} 
\circ g^{-d(M+1)}) = \supp (g_{M+1}) \subset \supp (f_{M+1}) = \supp 
(f_{M+1} \circ h^{-d(M+1)}) \subset G_{2} \subset U$, and $E \cap 
G_{2} = \varnothing$. Thus, the sets $\supp (g_{j} \circ h^{-d(j)})$ 
are pairwise disjoint subsets of $U$ for $0 \leq j \leq M+1$. It 
follows that $(X,h)$ has the dynamic comparison property.
\end{proof}

\begin{lem}\label{PosiErgodicAvgs}
Let $(X,h)$ be as in Notation \ref{MinDynSys}. Suppose that $F 
\subset X$ is closed and $E \subset X$ is open such that 
\begin{enumerate}
\item $F \cap \overline{E} = \varnothing$; 
\item $\mu (\del F) = 0$ and $\mu (\del E) = 0$ for all $\mu \in 
M_{h}(X)$;
\item $\mu (F) < \mu (E)$ for all $\mu \in M_{h}(X)$. 
\end{enumerate}
Then there exist continuous functions $g_{0},g_{1} 
\colon X \to [0,1]$ such that $g_{0} = 1$ on $F$, $\supp (g_{0}) 
\subset X \setminus \overline{E}$, $\supp (g_{1}) \subset E$, 
and
\[
\inf_{\mu \in M_{h}(X)} \int_{X} g_{1} -  g_{0} \; d\mu > 0.
\]
Moreover, with $g = g_{1} - g_{0}$, there exist $N_{0} \in \N$ and
$\sm > 0$ such that for all $N \geq N_{0}$ and $x \in X$, we have
\[ \frac{1}{N} \sum_{j=0}^{N-1} g(h^{j}(x)) \geq \sm. \]
\end{lem}

\begin{proof}
Since $F \cap \overline{E} = \varnothing$, the normality of $X$ gives
open sets $V_{0},V_{1} \subset X$ such that $F \subset V_{0}$,
$\overline{E} \subset V_{1}$, and $V_{0} \cap V_{1} = \varnothing$. Let
$\eps = \tfrac{1}{3} \inf_{\mu \in M_{h}(X)} [\mu (E) - \mu (F)] > 0$. 
Corollary \ref{MeasApproxZeroBdy} implies there exist and open set 
$W \subset X$ and a compact set $K \subset X$ such that $F \subset 
W$, $K \subset E$, $\mu (U \setminus F) < \eps$ for all $\mu \in 
M_{h}(X)$, and $\mu (E \setminus K) < \eps$ for all $\mu \in M_{h}(X)$. 
Set $W_{0} = V_{0} \cap W$, which satisfies $F \subset W_{0}$, 
$W_{0} \cap \overline{E} \subset W_{0} \cap V_{1} = \varnothing$, and 
$0 < \mu (W_{0} \setminus F) \leq \mu (W \setminus F) < \eps$ for all 
$\mu \in M_{h}(X)$. Now choose continuous functions  $g_{0}$ and 
$g_{1}$ such that $g_{0} = 1$ on $F$, $\supp (g_{0}) \subset W_{0}$, 
$g_{1}  = 1$ on $K$, and $\supp (g_{1}) \subset E$. Observing that
\begin{align*}
\mu (K) - \mu (W_{0}) &= ( \mu (E) - \mu (F) ) - ( \mu (E) - \mu
(K) ) - ( \mu (W_{0}) - \mu (F) ) \\
&= (\mu (E) - \mu (F)) - (\mu (E \setminus K)) - (\mu(W_{0} 
\setminus F)) \\
&> 3 \eps - \eps - \eps \\ 
&= \eps
\end{align*}
for all $\mu \in M_{h}(X)$, it follows that $\inf_{\mu \in M_{h}(X)} 
[\mu (K) - \mu (W)] \geq \eps > 0$. Since $g_{0} \leq \chi_{W_{0}}$ 
and $\chi_{K} \leq g_{1}$, we obtain
\[
\inf_{\mu \in M_{h}(X)} \int_{X} g_{1} - g_{0} \; d\mu \geq \inf_{\mu 
\in M_{h}(X)} \int_{X} \chi_{K} - \chi_{W_{0}} \; d\mu = \inf_{\mu \in 
M_{h}(X)} [ \mu(W_{0}) - \mu (K) ] \geq \eps > 0.
\]
Noting that, by the previous calculation, the function $g = g_{1} - 
g_{0}$ satisfies
\[ 
\inf_{\mu \in M_{h}(X)} \int_{X} g \; d\mu > 0, 
\]
we define $\sm > 0$ by
\[
\sm = \frac{1}{2} \inf_{\mu \in M_{h}(X)} \int_{X} g \; d\mu.
\]
Suppose that no $N_{0} \in \N$ as in the 
statement of the Lemma exists. Then there exist sequences 
$(N_{k})_{k=1}^{\infty} \subset \N$ and $(x_{k})_{k=1}^{\infty} 
\subset X$ such that for all $k \in \N$ we have
\[ 
\frac{1}{N_{k}} \sum_{j=0}^{N_{k}-1} g(h^{j}(x_{k})) \leq \sm. 
\]
Passing to subsequences $(N_{k(l)})_{l=1}^{\infty}$ and
$(x_{k(l)})_{l=1}^{\infty}$ (if necessary) and applying the
pointwise ergodic theorem (see the remark after Theorem 
1.14 of \cite{Wal}) yields
\[
\int_{X} g \; d\mu = \lim_{l \to \infty} \frac{1}{N_{k(l)}}
\sum_{j=0}^{N_{k(l)}-1} g(h^{j}(x_{k(l)})) \leq \sm
\]
for every $\mu \in M_{h}(X)$. We conclude that 
\[
\sup_{\mu \in M_{h}(X)} \int_{X} g \; d\mu \leq \sm = \frac{1}{2} 
\inf_{\mu \in M_{h}(X)} \int_{X} g \; d\mu,
\]

a contradiction.
\end{proof}

\begin{lem}\label{TopSmallRokhlinTower}
Suppose that $(X,h)$ has the topological small boundary property. 
Then for any $N \in \N$, there exists a closed set $Y \subset X$ 
such that $\sint (Y) \neq \varnothing$, $\overline{\sint (Y)} = Y$, 
$\del Y$ is topologically $h$-small, and the sets 
$Y,h(Y),\ldots,h^{N}(Y)$ are pairwise disjoint.
\end{lem}

\begin{proof}
Since the action of $h$ on $X$ is free, for $y \in X$ the iterates 
$y,h(y),\ldots,h^{N}(y)$ are all distinct elements of $X$. Choose 
pairwise disjoint open neighborhoods $W_{0},\ldots,W_{N}$ of 
these points, and set $W = \bigcap_{j=0}^{N} h^{-j}(W_{j})$. Then 
the iterates $W,h(W),\ldots,h^{N}(W)$ are pairwise disjoint. Let $F = 
\set{y}$ and $K = X \setminus W$. Apply the topological small 
boundary property to obtain open sets $U, V \subset X$ such that 
$F \subset U$, $K \subset V$, $\overline{U} \cap \overline{V} = 
\varnothing$, and $\del U$ is topologically $h$-small. Setting $Y = 
\overline{U}$, it follows that $\sint (Y) = U \neq \varnothing$, 
$\overline{\sint (Y)} = Y$, and $\del Y$ is topologically $h$-small. 
Finally, as $X \setminus W \subset V$ and $Y \cap \overline{V} = 
\varnothing$, it follows that $Y \subset W$, and hence the sets 
$Y,h(Y),\ldots,h^{N}(Y)$ are pairwise disjoint.
\end{proof}

\begin{lem}\label{ThinTowerBoundaries}
Let $(X,h)$ be as in Notation \ref{MinDynSys}. Let $Y \subset X$ be 
closed with $\sint (Y) \neq \varnothing$ and $\del Y$ topologically 
$h$-small. Adopt the notation of Theorem \ref{RokhlinTower}. Then 
$\del (h^{j}(Y_{k}))$ is thin for $0 \leq k \leq l$ and $0 \leq j 
\leq n(k) - 1$.
\end{lem}

\begin{proof}
By Proposition \ref{TopSmallIsThin}, $\del Y$ is thin. For $0 \leq j 
\leq n(k) - 1$, we have $\del h^{j}(Y_{k}) = h^{j}(\del Y_{k})$, and 
since translates of thin sets are thin, it suffices to prove that 
each of the sets $\del Y_{k}$ is thin. But $\del Y_{k} \subset 
\bigcup_{j=0}^{n(l)-1} h^{j}(\del Y)$, and this set is thin by Lemma 
\ref{ThinUnionThin}, since it is a finite union of translates of 
thin sets.
\end{proof}

\begin{thm}\label{TopSmallDCP}
Let $(X,h)$ be as in Notation \ref{MinDynSys}, and suppose that 
$(X,h)$ has the topological small boundary property. Then $(X,h)$ 
has the dynamic comparison property.
\end{thm}

\begin{proof}
Let $C \subset X$ be closed and $U \subset X$ be open such that
$\mu (C) < \mu (U)$ for all $\mu \in M_{h}(X)$. By Lemma 
\ref{DCPSimplification}, we may assume that $\sint (C) \neq 
\varnothing$, $\del C$ is topologically $h$-small, and that there is 
an open set $U_{0} \subset U$ such that $\overline{U}_{0} \subset 
U$, $\del U_{0}$ is topologically $h$-small, $\overline{U}_{0} \cap 
C = \varnothing$, and $\mu (C) < \mu (U_{0})$ for all $\mu \in 
M_{h}(X)$. Applying Lemma \ref{PosiErgodicAvgs} to $C$ and 
$U_{0}$, there exist continuous functions $g_{0},g_{1} \colon X \to 
[0,1]$ such that $g_{0} = 1$ on $C$, $\supp (g_{0}) \subset X 
\setminus \overline{U}_{0}$, $\supp (g_{1}) \subset U_{0}$, and
\[
\inf_{\mu \in M_{h}(X)} \int_{X} g_{1} - g_{0} \; d\mu > 0.
\]
Moreover, with $g = g_{1} - g_{0}$, there exists $N_{0} \in \N$ and
$\sm > 0$ such that for all $N \geq N_{0}$ and $x \in X$, we have
\[
\frac{1}{N} \sum_{j=0}^{N-1} g(h^{j}(x)) \geq \sm.
\]

By Lemma \ref{TopSmallRokhlinTower}, there exists a closed set 
$Y \subset X$ with $\sint (Y) \neq \varnothing$ such that $\del Y$ is 
topologically $h$-small, and such that the sets $Y,h(Y),\ldots,
h^{N_{0}}(Y)$ are pairwise disjoint. Following the
notation of Theorem \ref{RokhlinTower}, we construct the Rokhlin
tower over $Y$ by first return times to $Y$, then apply the second
statement of Theorem \ref{RokhlinTower} with the partition
$\mathcal{P} = \set{U_{0},C,X \setminus (U_{0} \cup C)}$ of $X$ by
sets with non-empty interior (discarding the third set if it is
empty). For convenience, we will use $Y_{0},\ldots,Y_{l}$ and
$n(0) \leq n(1) \leq \cdots \leq n(l)$ for the base spaces and first
return times in the tower compatible with $\mathcal{P}$, and set 
$Y_{k}^{(0)} = Y_{k} \setminus \del Y_{k}$. (Note that since these 
$Y_{k}$ are the sets $Z_{k}$ in Theorem \ref{RokhlinTower}, it may 
be the case that $Y_{k}^{(0)} = \varnothing$.) We set
\[
F = X \setminus \left( \bigcup_{k=0}^{l} \bigcup_{j=0}^{n(k)-1}
h^{j} (Y_{k}^{(0)}) \right).
\]
For each $k \in \set{0,\ldots,l}$, the column $\set{ h^{j}(Y_{k}) 
\colon 0 \leq j \leq n(k) - 1}$ has height at least $N_{0}$. Thus, 
for any $x \in Y_{k}$ we have
\[
\frac{1}{n(k)} \sum_{j=0}^{n(k)-1} g(h^{j}(x)) \geq \sm > 0.
\]
For $S \subset X$ and $k \in \set{0,\ldots,l}$ define
\[
N(S,k) = \set{n \in \set{0,1,\ldots,n(k)-1} \colon h^{n} (Y_{k}) 
\subset S}.
\]
Letting $\chi = \chi_{U_{0}} - \chi_{C}$, we observe that $g_{0} 
= 1$ on $C$ implies that $\chi_{C} \leq g_{0}$ and $\supp (g_{1}) 
\subset U_{0}$ implies that $g_{1} \leq \chi_{U_{0}}$. Combining 
these inequalities gives $g \leq \chi$, and so
\[
0 < \sm \leq \frac{1}{n(k)} \sum_{j=0}^{n(k)} g(h^{j}(x))
\leq \frac{1}{n(k)} \sum_{j=0}^{n(k)} \chi (h^{j}(x)) =
\frac{ \card (N(U_{0},k)) - \card (N(C,k))}{n(k)}.
\]
It follows that for $0 \leq k \leq l$, we have $\card (N(U_{0},k)) >
\card (N(C,k))$ (that is, more levels in the column $\set{ h^{j}
(Y_{k}) \colon 0 \leq j \leq n(k) - 1}$ are contained in $U_{0}$ 
than are contained in $C$) and so there is an injective map $\ph_{k} 
\colon N(C,k) \to N(U_{0},k)$. If we order $N(C,k)$ as $\set{s_{k}(0)
,\ldots,s_{k}(L_{k})}$ and order $N(U_{0},k)$ as 
$\set{t_{k}(0),\ldots,t_{k}(L_{k}),\ldots}$, then one way to 
represent the injection $\ph_{k}$ is by $\ph_{k} = (d_{k}(0),\ldots,
d_{k}(L_{k})) \in \Z^{L_{k}}$ where, for $0 \leq m \leq L_{k}$, the 
integer $d_{k}(m)$ satisfies
\[
h^{d_{k}(m)} ( h^{s_{k}(m)} (Y_{k}) ) \subset h^{t_{k}(m)} (Y_{k}).
\]

Next, we claim that the closed set $F$ is thin. Since the finite 
union of thin sets is thin by Lemma \ref{ThinUnionThin}, it clearly 
suffices to prove that $\del h^{j}(Y_{k})$ is thin for each $0 \leq 
k \leq l$, $0 \leq j \leq n(k) - 1$. Now, $\del C$ and $\del U_{0}$ 
are both topologically $h$-small, hence thin. Since $\del (X 
\setminus (U_{0} \cup C) ) = \del (U_{0} \cup C) \subset \del U_{0} 
\cup \del C$, it follows that the boundaries of all sets in the 
partition $\mathcal{P}$ are thin. As the only processes used in the 
construction of the Rokhlin tower compatible with this partition are 
translation by powers of $h$, finite unions, and finite 
intersections, it follows that it is sufficient to prove that the 
boundaries $\del h^{j}(Y_{k})$ in a standard Rokhlin tower (without 
any condition about compatibility with respect to a partition) are 
thin. This is true by Lemma \ref{ThinTowerBoundaries}, and 
consequently $F$ is thin. 

Now, set $Q = \set{k \colon 0 \leq k \leq l, Y_{k}^{(0)} \neq 
\varnothing}$, $Q' = \set{0,\ldots,l} \setminus Q$, and define 
\[
\eps = \ts{\frac{1}{2}} \min_{k \in Q} \inf_{\mu \in M_{h}(X)} \mu 
(Y_{k}^{(0)}) > 0.
\]
Apply Lemma \ref{LeftoverCoverSpace} 
with $F$, $U \setminus \overline{U}_{0}$, and $\eps$. We obtain $M 
\in \N$, and for $0 \leq i \leq M$ open sets $T_{i},V_{i},W_{i} 
\subset X$, closed sets $F_{i} \subset X$, continuous functions 
$b_{i} \colon X \to [0,1]$, and integers $r(i)$ such that:
\begin{enumerate}
\item $h^{-r(i)}(F_{i}) \subset T_{i} \subset \overline{T}_{i}
\subset V_{i} \subset \overline{V}_{i} \subset W_{i} \subset U
\setminus \overline{U}_{0}$ for $0 \leq i \leq M$;
\item $\sum_{i=0}^{M} b_{i} = 1$ on $\bigcup_{i=0}^{M} h^{r(i)}
( \overline{V}_{i} )$;
\item $\supp (b_{i} \circ h^{-r(i)}) \subset W_{i}$ for $0 \leq i
\leq M$;
\item the sets $W_{i}$ are pairwise disjoint and $\sum_{i=0}^{M}
\mu (W_{i}) < \eps$ for all $\mu \in M_{h}(X)$.
\end{enumerate}
By the choice of $\eps$, it follows that for $k \in Q$ and $0 \leq 
j \leq n(k) - 1$, we have 
\begin{align*}
\mu \left( h^{j}(Y_{k}^{(0)}) \setminus \bigcup_{i=0}^{M} h^{r(i)}
(W_{i}) \right) &\geq \mu ( h^{j} (Y_{k}^{(0)}) ) - \mu \left(
\bigcup_{i=0}^{M} h^{r(i)} (W_{i}) \right) \\
&\geq 2 \eps - \sum_{i=0}^{M} \mu (h^{r(i)} (W_{i})) \\
&= 2 \eps - \sum_{i=0}^{M} \mu (W_{i}) \\ 
&> \eps
\end{align*}
for every $\mu \in M_{h}(X)$, and therefore 
\[
\inf_{\mu \in M_{h}(X)} \mu \left( h^{j}(Y_{k}^{(0)}) \setminus 
\bigcup_{i=0}^{M} h^{r(i)} (W_{i}) \right) \geq \eps > 0.
\]
This implies the sets $h^{j}(Y_{k}^{(0)}) \setminus \bigcup_{i=0}^{M}
h^{r(i)} (W_{i})$ are non-empty whenever $k \in Q$. It follows that 
for $k \in Q$, each set $h^{j}(Y_{k}) \setminus \bigcup_{i=0}^{M} 
h^{r(i)} (V_{i})$ is a non-empty closed subset of $h^{j} (Y_{k})$. 
Now for $k \in Q$ and $0 \leq m \leq L_{k}$ choose a continuous 
function $f_{m,k} \colon X \to [0,1]$ such that $f_{m,k} = 1$ on 
$h^{s_{k}(m)} (Y_{k}) \setminus \bigcup_{i=0}^{M} h^{r(i)} 
(V_{i})$ and $\supp (f_{m,k}) \subset h^{s_{k}(m)} (Y_{k}) \setminus 
\bigcup_{i=0}^{M} h^{r(i)} (\overline{T}_{i})$. Now we have 
collections of continuous functions
\[
\set{b_{i} \colon 0 \leq i \leq M} \cup \set{ f_{m,k} \colon k \in 
Q, 0 \leq m \leq L_{k} }
\]
and associated integers
\[
\set{ r(i) \colon 0 \leq i \leq M} \cup \set{ d_{k}(m) \colon k \in 
Q, 0 \leq m \leq L_{k} }.
\]
For any $x \in C$, if $x \in \bigcup_{k \in Q} \bigcup_{m=0}^{L_{k}}
\left( h^{s_{k}(m)} (Y_{k}) \setminus \bigcup_{i=0}^{M} h^{r(i)}
(V_{i}) \right)$, then $f_{m,k} (x) \neq 0$ for some $k \in Q$ and 
some $m \in \set{0,\ldots,L_{i,k}}$. Otherwise, $x \in
\bigcup_{i=0}^{M} h^{r(i)}(V_{i})$, and $b_{i} (x) \neq 0$ for some
$0 \leq i \leq M$. (Notice that if $x \in \bigcup_{k \in Q'} 
\bigcup_{m=0}^{L_{k}} h^{s_{k}(m)} (Y_{k})$, then in 
fact $x \in F$, and so also $x \in \bigcup_{i=0}^{M} h^{r(i)} 
(V_{i})$.) Now re-order the two collections above as $\set{
f_{j}^{(0)} \colon 0 \leq j \leq K}$ and $\set{d(j) \colon 0 \leq j
\leq K}$ for an appropriate $K \in \N$. Then $\sum_{j=0}^{K}
f_{j}^{(0)} (x) > 0$ for all $x \in C$. Since $C$ is compact and 
the $f_{j}^{(0)}$ are continuous, there must be a $\om > 0$ such 
that $\sum_{j=0}^{K} f_{j}^{(0)}(x) \geq \om$ for all $x \in C$. 
Again using continuity, we can choose an open set $S \subset X$ such 
that $C \subset S$ and $\sum_{j=0}^{K} f_{j}^{(0)} (x) \geq 
\ts{\frac{1}{2}} \om$ for all $x \in S$. Choose a continuous function 
$f \colon X \to [0,1]$ such that $f(x) = 1$ for all $x \in C$, and 
$\supp (f) \subset S$. For $0 \leq j \leq K$ define continuous 
functions $f_{j} \colon X \to [0,1]$ by 
\[
f_{j} (x) = \begin{cases} f(x) f_{j}^{(0)}(x) \left( \sum_{i=0}^{K} 
f_{i}^{(0)} (x) \right)^{-1} & \textnormal{if} \; x \in S \\ 0 & 
\textnormal{if} \; x \not \in S. \end{cases}
\]
Then for any $x \in C$, 
\[
\sum_{j=0}^{K} f_{j}(x) = \left( \sum_{i=0}^{K} f_{i}^{(0)} (x) 
\right)^{-1} \sum_{j=0}^{K} f(x) f_{j}^{(0)}(x) = \left( 
\sum_{i=0}^{K} f_{i}^{(0)} (x) \right)^{-1} \sum_{j=0}^{K} 
f_{j}^{(0)}(x) = 1.
\]
Moreover, $\supp (f_{j}) \subset \supp (f_{j}^{(0)})$ for $0 \leq j 
\leq K$. If $f_{j}^{(0)} = b_{i}$ for some $0 \leq i \leq M$, then
\[
\supp (f_{j}^{(0)} \circ h^{-d(j)}) = \supp (b_{i} \circ h^{-r(i)})
\subset W_{i} \subset U \setminus \overline{U}_{0}
\]
and the sets $W_{i}$ are pairwise disjoint. Therefore the sets $\supp
(f_{j}^{(0)} \circ h^{-d(j)})$ are pairwise disjoint for all choices
of $j$ where $f_{j}^{(0)} \in \set{b_{i} \colon 0 \leq i \leq M}$.
Next, if $f_{j}^{(0)} = f_{m,k}$ for some $k \in Q$ and some $0 \leq 
m \leq L_{k}$, then
\[
\supp (f_{j}^{(0)} \circ h^{-d(j)}) = \supp (f_{m,k} \circ
h^{-d_{k}(m)}) \subset h^{t_{k}(m)} (Y_{k}) \subset U_{0}.
\]
Moreover, the definition of the functions $f_{m,k}$ implies that 
\[
\supp (f_{m,k} \circ h^{-d_{k}(m)}) \subset h^{d(m)} \left( 
h^{s_{k}(m)}(Y_{k}) \setminus \ts{\bigcup_{i=0}^{M} h^{r(i)}(V_{i})} 
\right),
\]
so that in particular, for $k \in Q$ the set $\supp (f_{m,k} \circ 
h^{-d_{k}(m)})$ is a subset of $h^{t_{k}(m)}(Y_{k}^{(0)})$ (which is 
non-empty by the choice of $k$). Since the sets $h^{t_{k}(m)} 
(Y_{k}^{(0)})$ are pairwise disjoint, the sets $\supp (f_{j}^{(0)} 
\circ h^{-d(j)})$ are pairwise disjoint for all choices of $j$ where 
$f_{j}^{(0)} \in \set{ f_{m,k} \colon k \in Q, 0 \leq m \leq 
L_{k}}$. Moreover, the sets $W_{i}$ are pairwise disjoint from 
the sets $h^{t_{k}(m)}(Y_{k}^{(0)})$, as $U \setminus 
\overline{U}_{0}$ is certainly disjoint from $U_{0}$. Therefore, the 
sets $\supp (f_{j}^{(0)} \circ h^{-d(j)})$ are pairwise disjoint 
subsets of $U$ for all $0 \leq j \leq K$. It follows that the sets 
$\supp (f_{j} \circ h^{-d(j)})$ are pairwise disjoint subsets of $U$ 
for all $0 \leq j \leq K$. This completes the proof.
\end{proof}

\begin{cor}\label{DCPforFinDim}
Let $(X,h)$ be a minimal dynamical system, where $\dim (X) < 
\infty$. Then $(X,h)$ has the dynamic comparison property.
\end{cor}

\begin{proof}
By Proposition \ref{FinDimTSBP}, $(X,h)$ has the topological small 
boundary property. Theorem \ref{TopSmallDCP} then implies 
$(X,h)$ has the dynamic comparison property.
\end{proof}

In \cite{Buck1} and \cite{Buck2}, the dynamic comparison property will 
be used to study properties of crossed product $C^{*}$-algebras 
$C^{*}(\Z,C(X,A),\bt)$, where $\bt$ is an automorphism whose 
restriction to $C(X)$ is the action induced by a minimal 
homeomorphism. In particular, it plays a key role in showing that given 
a nonzero positive element of $C(X,A)$, there is a non-zero positive 
element of $C(X)$ which is Cuntz subequivalent to it. Taylor Hines has 
suggested that the dynamic comparison property may be a topological 
analogue for strict comparison of positive elements. For this to be 
reasonable, something like the following should hold:

\begin{cnj}\label{DCPimpliesSCPE}
Let $(X,h)$ be as in Notation \ref{MinDynSys}, and assume that $(X,h)$ 
has the dynamic comparison property. Then the transformation group 
$C^{*}$-algebra $C^{*}(\Z,X,h)$ has strict comparison of positive 
elements.
\end{cnj}

This result would be independent of the proof in \cite{TomsWinter} that 
$C^{*}(\Z,X,h)$ is $\JS$-stable for finite-dimensional $X$, since the 
dynamic comparison property holds for infinite-dimensional systems 
such as that in Corollary \ref{HinesExample}. It would also verify a 
conjecture of Toms and Phillips (namely, that the radius of comparison 
for $C(X)$ is approximately half the mean dimension of $X$) for all 
systems with the dynamic comparison property (including all those 
with the topological small boundary property).

\end{document}